\documentclass[12pt]{amsart}
\usepackage[american]{babel}
\usepackage{amsmath}
\usepackage{latexsym}
\usepackage{amssymb}
\usepackage{enumerate}
\usepackage{mathtools}
\DeclareMathAlphabet\mathbfcal{OMS}{cmsy}{b}{n}
\theoremstyle{plain}
\newtheorem{Thm}[subsection]{Theorem}
\newtheorem{Cor}[subsection]{Corollary}
\newtheorem{Prop}[subsection]{Proposition}
\newtheorem{Lem}[subsection]{Lemma}

\theoremstyle{definition}

\newtheorem{Def}[subsection]{Definition}

\renewcommand{\phi}{\varphi}
\newcommand{\RR}{\mathbb{R}}
\newcommand{\CC}{\mathbb{C}}

\newcommand{\QQ}{\mathbb{Q}}

\renewcommand{\emptyset}{\varnothing}
\renewcommand{\setminus}{-}
\newcommand{\im}{\operatorname{im}}

\begin{document}

\title[Automatic continuity]{Automatic continuity of abstract homomorphisms between
                  locally compact and Polish groups}
                  
%

\author{Oskar Braun, Karl H. Hofmann and Linus Kramer}
\address{Oskar Braun, Linus Kramer\newline\indent
Mathematisches Institut, Universit\"at M\"unster
\newline\indent
Einsteinstr. 62, 
48149 M\"unster,
Germany}
\email{linus.kramer{@}uni-muenster.de}
\thanks{Partially supported by SFB 878}
\address{Karl H. Hofmann\newline\indent
Fachbereich Mathematik, Technische Universit\"at Darmstadt \newline\indent
Schlossgartenstra{\ss}e 7,
64289 Darmstadt, Germany} 
\email{hofmann{@}mathematik.tu-darmstadt.de}

\maketitle

\begin{abstract}
We prove results about automatic continuity and openness of abstract surjective group homomorphisms
$K\xrightarrow{\ \phi\ } G$, where $G$ and $K$ belong to a certain class $\mathbfcal K$ of topological groups,
and where the kernel of $\phi$ satisfies a certain topological countability condition.
Our results apply in particular to the case where $G$ is a semisimple Lie group or a semisimple 
compact group, and where $\mathbfcal K$ is either the class of all locally compact groups or the class 
of all Polish groups.
\end{abstract}

\section*{Introduction}
We are concerned with questions of the following type. Suppose that 
$G$ and $K$ are topological groups belonging to a certain class $\mathbfcal K$ of spaces, and suppose that
\[
 K\xrightarrow{\ \phi\ } G
\]
is an abstract (i.e. not necessarily continuous)
surjective group homomorphism. Under what conditions
on the group $G$ and the kernel $\operatorname{ker}(\phi)$ is the homomorphism
$\phi$ automatically continuous and open?
Questions of this type have a long history and were studied in particular 
for the case that $G$ and $K$ are Lie groups, compact groups, or Polish groups.

We develop an axiomatic approach, which allows us to resolve the question 
uniformly for different classes of topological groups.
In this way we are able to extend the classical results about automatic 
continuity to a much more general setting.
We shall say that a class $\mathbfcal K$ 
of topological Hausdorff spaces which is closed under the passage to closed subspaces
and closed under finite products is
\emph{almost Polish} if every nonempty space $X$ in the class satisfies the
following properties:
\begin{enumerate}[(1)]
\item  Every open covering of $X$ has a countable subcovering.
\item  The space $X$ is not a countable union of nowhere dense subsets.
\item  For each continuous image $A\subseteq X$ of some $\mathbfcal K$-member 
       there is an open set $U\subseteq X$ such that the symmetric
       difference  $(A-U)\cup(U-A)$  is a countable union 
       of nowhere dense subsets of $X$.
\end{enumerate}

\noindent
Note that by (3), being almost Polish is a property of a class of spaces,
and not a property of an individual space.
The class $\mathbfcal P$ of Polish spaces, the class 
$\mathbfcal L^\sigma$ of locally compact $\sigma$-compact spaces
and the class $\mathbfcal C$ of compact spaces are almost Polish.
A more systematic discussion of 
almost Polish classes of spaces will be presented 
in Section~\ref{TopologySection} below.

Given an almost Polish class $\mathbfcal K$ and a space $X\in\mathbfcal K$,
we call a subset $A\subseteq X$ a \emph{$\mathbfcal K$-analytic set}
if $A=\psi(Z)$ holds for some $Z\in\mathbfcal K$ and
some continuous map $\psi:Z\longrightarrow X$. This notion parallels
the notion of an analytic set (or Suslin set) in the classical theory
of Polish spaces and in descriptive set theory. 
Singletons are $\mathbfcal K$-analytic
in every almost Polish class $\mathbfcal K$.

Suppose that $\mathbfcal K$ is an almost Polish class of spaces, and that
$G$ is a topological group that belongs to $\mathbfcal K$.
We now introduce a piece of terminology which we shall use
throughout this text. We shall call
 $G$ \emph{rigid within the class} $\mathbfcal K$,
or $\mathbfcal K$-\emph{rigid} for short, if the following holds.

\smallskip\noindent\textbf{Rigidity.}
{\em For every short exact sequence of groups
\[
 1\longrightarrow N \xhookrightarrow{\ \ \ } K\xrightarrow{\ \phi\ }G\longrightarrow 1,
\]
where $\phi$ is an abstract group homomorphism,
where  $G$ and $K$ belong to $\mathbfcal K$, and where
the kernel $N$ of $\phi$ is $\mathbfcal K$-analytic,
the homomorphism $\phi$ is automatically continuous and open.}

\smallskip
Notice right away that any topological group $G$ supporting a
discontinuous automorphism fails to be $\mathbfcal K$-rigid for
any almost Polish class $\mathbfcal K$ containing $G$.
The following consequence of rigidity is immediate from the definition.

\smallskip\noindent
\textbf{Uniqueness of group topologies.}
{\em A $\mathbfcal K$-rigid topological group $G$ for
an almost Polish class $\mathbfcal K$ has a unique group topology 
in that class $\mathbfcal K$.
In particular,
every abstract group automorphism of $G$ is a homeomorphism.}

\smallskip
Before we formulate our first
rigidity result   we need to recall that
a simple real Lie algebra $\mathfrak g$ is called \emph{absolutely simple} if its
complexification $\mathfrak g\otimes_{\mathbb R}\mathbb C$ is a simple complex Lie algebra.
The simple real Lie algebras which are not absolutely simple are precisely the complex
simple Lie algebras, viewed as real Lie algebras.
For example, the real Lie algebra $\mathfrak{so}_{1,3}\RR$ is not absolutely simple, because 
it is isomorphic as a real Lie algebra to~$\mathfrak{sl}_2\CC$.

\smallskip\noindent\textbf{Theorem A.}
{\em A Lie group  $G$ is rigid within every 
almost Polish class $\mathbfcal K$ containing it, provided it satisfies
the following conditions:
\begin{enumerate}[\rm(1)]
\item The center $\operatorname{Cen}(G^\circ)$ of its identity component  is finite.
\item Its Lie algebra $\operatorname{Lie}(G)$ is a direct sum 
      of absolutely simple ideals.
\end{enumerate}  
}

\smallskip
See Theorem~\ref{SemisimpleThm} for the proof. The case of semisimple Lie groups with infinite
centers, like $\widetilde{\mathrm{SL}_2\RR}$, remains open.
Theorem A generalizes results in  \cite{BT}, \cite{Car}, \cite{Freu}, \cite{Tits}, and \cite{vdW},
which mainly concern automatic continuity of abstract isomorphisms within the class of 
semisimple Lie groups, and 
\cite[5.66]{HMCompact}, \cite{Ka74}, \cite{Kal82}, and \cite{KramerUnique}
which concern rigidity with respect to subclasses of $\mathbfcal P$ and $\mathbfcal L^\sigma$.
We also prove rigidity results for certain semidirect products of vector groups and
classical Lie groups in Section~\ref{SemidirectSection}.
We refer to Theorem~\ref{SemidirectTheorem} and Theorem~\ref{ClassicalGroups}.

Following \cite[9.5]{HMCompact}, we shall call a compact topological group $G$ \emph{semisimple}
if it is connected and perfect. 

\smallskip\noindent\textbf{Theorem B.}
{\em A compact semisimple group $G$ is rigid within every almost Polish class $\mathbfcal K$ containing it.}

\smallskip
In Theorem \ref{CompactSemisimpleThm} below we in fact prove a rigidity result which holds for a much larger
class of compact groups than semisimple groups, including many profinite groups. 
For the class $\mathbfcal L^\sigma$, Theorem~B is essentially proved in \cite{Braun}, 
and in a more restricted form in~\cite{Stewart}.

In a different direction, and building on work by Nikolov--Segal, we obtain the following result,
which generalizes \cite{Pejic}. See Theorem~\ref{ProfiniteTheorem}.

\smallskip\noindent\textbf{Theorem C.}
{\em A topologically finitely generated profinite group $G$ is rigid within every almost Polish class $\mathbfcal K$ 
containing it.}

\smallskip
All these rigidity results deal with abstract homomorphisms $K\longrightarrow G$ where the \emph{range}
$G$ has prescribed properties. However, the methods that we develop are also
capable of producing automatic continuity results for abstract homomorphisms
$G\longrightarrow H$ between topological groups where the \emph{domain} $G$ is
a Lie group with special properties.

\smallskip

For this purpose,
we call a subset $C$ of a topological group 
\emph{spacious}
if some product of finitely many left translates of $CC^{-1}$ has nonempty interior.
Small spacious sets abound in many Lie groups, as we shall
show in Section~\ref{YamabeBaireSection}.
Theorem D generalizes \cite[5.64]{HMCompact} and~\cite{vdW}, see Theorem~\ref{OtherDircetion} and
its corollaries.

\smallskip\noindent\textbf{Theorem D.}
{\em Let $\psi: G\longrightarrow H$ be an abstract homomorphism from
a Lie group $G$ into a topological group $H$ satisfying the following conditions:
\begin{enumerate}[\rm(1)] 
\item The Lie algebra $\operatorname{Lie}(G)$ is  perfect.
\item  There exists a compact spacious set $C\subseteq G$ 
       whose image $\psi(C)\subseteq H$ has compact closure.
\end{enumerate}
Then $\psi$ is continuous.}

\smallskip
Related results concerning automatic continuity for SIN groups are proved in \cite{DowerkThom} and in \cite{DielsDowerk}.
We also take the opportunity to correct a mistake which occurred both in \cite{KramerUnique} 
and in~\cite{Braun}, see Section~\ref{Erratum} below.

\section*{Acknowledgment}
The authors thank Karl-Hermann Neeb for pointing out several useful references.
The third author thanks Walter Neumann for his hospitality at Columbia University,
and Eugen Hellmann for a helpful discussion of the fields $\CC_p$. 
We also thank the referees for helpful remarks and constructive comments.

\section*{The strategy}
Suppose that we are given a short exact sequence
\[
1\longrightarrow N \xhookrightarrow{\ \ \ } K\xrightarrow{\ \phi\ }G\longrightarrow 1,
\]
where $G,K$ are groups in some almost Polish class $\mathbfcal K$, where
$\phi$ is an abstract group homomorphism, and where $N\unlhd K$ is a normal (not necessarily closed)
subgroup which is a $\mathbfcal K$-analytic set. 
As a first step, we show that $\phi$ is continuous and open, provided that we can construct
a neighborhood basis of the identity in $G$ consisting of sets $V\subseteq G$ whose
preimages $\phi^{-1}(V)$ are $\mathbfcal K$-analytic. This is the main result in
Section~\ref{TopologySection}.

The actual construction of this neighborhood basis depends on the nature of the group $G$.
If $G$ is an $n$-dimensional Lie group whose Lie algebra is perfect, then we show the following.
Suppose that $D\subseteq G$ is one fixed compact identity neighborhood. Then there  exist
$1$-parameter subgroups $c_1,\ldots,c_n$ in $G$ such that the family
$M_{t_1,\ldots,t_n}=[c_1(t_1),D]\cdots [c_n(t_n),D]$, where $0<t_i\leq 1$ holds for all $i=1,\ldots,n$,
is a neighborhood basis of the identity in $G$. This observation goes back to van der Waerden.

Suppose that $G$ is a compact connected semisimple Lie group. In this case we may put
$D=G$ and then apply van der Waerden's construction. 
It turns out that then the sets $\phi^{-1}(M_{t_1,\ldots,t_n})$ are $\mathbfcal K$-analytic.
In view of the results in Section~\ref{TopologySection}, this allows us to conclude that
$\phi$ is continuous and open.

If $G$ is a connected semisimple Lie group with finite center, but not compact, 
then harder work is required.  
In this case we show that there exists a finite set $X\subseteq G$ and an element
$h\in G$ such that the set $C=\{ghg^{-1}\mid g\in\operatorname{Cen}_G(X)\}$ is compact,
and such that there exist elements $g_1,\ldots,g_r\in G$ such that the set
$D=g_1CC^{-1}g_2CC^{-1}\cdots g_{r}CC^{-1}\subseteq G$ is a compact identity neighborhood.
This result depends on the advanced structure theory of real semisimple Lie groups.
It requires that each simple ideal in the Lie algebra of $G$ 
is absolutely simple.
We also use Yamabe's Theorem saying that a path connected subgroup
of a Lie group is an analytic Lie subgroup,
and Baire's Category Theorem. In any case,
it follows again that the sets $\phi^{-1}(M_{t_1,\ldots,t_n})$ are $\mathbfcal K$-analytic
in $K$,
and by the results in Section~\ref{TopologySection}, the map $\phi$ is continuous and open.

For general compact connected semisimple groups and for topologically finitely generated profinite groups, the arguments
are somewhat different. However, they always boil down to the construction of a
neighborhood basis of the identity in $G$ consisting of sets that arise in an 
`algebraic' way starting from a finite set of group elements.

\section{Examples and counterexamples}
\label{ExampleSection}
Before we embark on the proofs of our main results, we collect a series of examples
which illustrate that our main results fail if certain assumptions are dropped.

\subsection{Rigidity fails for abelian groups.}\label{Ex1}
As abstract groups, the connected Lie groups $\mathbb C^*$ and $\operatorname{U}(1)\cong\mathbb R/\mathbb Z$ are isomorphic
\cite[A1.43]{HMCompact}.
Hence uniqueness of topologies fails for these locally compact, Polish, abelian groups. Also, 
the compact group $\mathbb R/\mathbb Z$ has non-continuous abstract automorphisms.

\subsection{Rigidity fails for 
groups which are not locally compact or $\sigma$-compact.}\label{Ex2}
By~\cite{Kiltinen}, the field 
$\mathbb{R}$ admits uncountably many nondiscrete
field topologies, which are not locally compact. 
Let $\mathcal T$ be one such topology which is
different from the usual topology $\mathcal S$ on the reals. Let $\mathcal T'$ and $\mathcal S'$ denote the 
topologies on the matrix group
$\operatorname{SO}(3)\subseteq\mathbb R^{3\times 3}$ induced by $\mathcal T$ and $\mathcal S$.
Then the identity map $(\operatorname{SO}(3),\mathcal T')\longrightarrow(\operatorname{SO}(3),\mathcal S')$
is not continuous. Such a topology $\mathcal T'$ on $\operatorname{SO}(3)$ fails to be locally compact or Polish.
This follows, for example, from our Theorem $D$. 

The discrete topology $\mathcal D$ on $\operatorname{SO}(3)$ is locally compact
and metrizable,
but neither $\sigma$-compact nor Polish. The identity map 
$(\operatorname{SO}(3),\mathcal D)\to(\operatorname{SO}(3),\mathcal S')$
is a continuous
bijective homomorphism which is not open.

\subsection{Rigidity fails for infinite products of compact 
Lie groups if the kernel is not restricted.}\label{Ex3}
(See \cite[pp.~182--183]{HeHoMo}.)
Let $I$ be an infinite set, let $G$ be a compact group and let 
\[K=\prod_{i\in I}G=\operatorname{Map}(I,G).\]
Let $I\xhookrightarrow{\ \ \ }\beta I$ denote the \v Cech--Stone compactification of the discrete space $I$.
For $x\in\beta I$ put $\mu(x)=\{J\subseteq I\mid x\in\overline J\}$. Then $\mu(x)$ is an ultrafilter
on $I$ which is free if and only if $x\in\beta I\setminus I$.
From the universal property of the \v Cech--Stone compactification we obtain a bijection
\[
 \operatorname{Map}(I,G)\xhookrightarrow{\qquad} C(\beta I,G),
\]
where $C(\beta I,G).$ denote the set of continuous maps from $\beta I$ to $G$.
For $x\in\beta I$, the evaluation homomorphism $x^*:C(\beta I,G)\longrightarrow G$, $f\longmapsto f(x)$
is surjective, because $G$ embeds diagonally in $\operatorname{Map}(I,G)\subseteq C(\beta I,G)$ as the set
of constant maps. If $x=j\in I$, then
$x^*=\operatorname{pr}_j$ is the projection onto the $j$th coordinate, which is continuous and 
open. However if $x\in\beta I\setminus I$, then the kernel of $x^*$ is dense in $K$,
as can be seen from the fact that then the ultrafilter $\mu(x)$ contains all cofinite 
sets.\footnote{\ From a different viewpoint, we express the group $G$ here 
via $x^*$ as an \emph{ultralimit}
or \emph{asymptotic cone} of a constant family $\{G\}$ of compact metric groups, without rescaling.}

Suppose that $I$ is countably infinite and that $G=\operatorname{Alt}(5)$ 
or that $G=\operatorname{SO}(3)$. Then $G$ and $K$ are compact Polish groups and
$x^*:K\longrightarrow G$ is surjective, but not continuous if $x\in\beta I\setminus I$.
This shows that we have to make some topological assumption on the kernel of the homomorphism
in order to obtain rigidity results.

In view of Theorem B and its generalization, Theorem~\ref{CompactSemisimpleThm},
we see also that $\operatorname{ker}(x^*)$ is then neither $\sigma$-compact
nor an analytic set. In the case that $G=\operatorname{Alt}(5)$ we can even conclude that 
the finite index subgroup $\operatorname{ker}(x^*)\subseteq K$
is not Haar measurable, since otherwise $\operatorname{ker}(x^*)$ would necessarily have positive
and finite volume, and hence would be open and closed.

\subsection{Rigidity fails for $\operatorname{SL}_n(\mathbb{C})$ and all
infinite complex linear algebraic groups.}\label{Ex4}
The field of complex numbers has $2^{2^{\aleph_0}}$ non-continuous automorphisms.
Each of these automorphisms extends entry-wise to a non-continuous automorphism of the matrix
group $\operatorname{SL}_n(\mathbb C)$. More generally, this method gives
non-continuous automorphisms of all infinite complex matrix groups $\mathbf{G}(\mathbb C)$,
where $\mathbf G$ is a linear algebraic group defined over $\mathbb C$. Such a group
$\mathbf G(\mathbb C)$ is in a natural way a complex Lie group.

But there are even more topologies on the field $\mathbb C$. Let $p$ be a prime and 
let $\mathbb C_p$ denote the completion of the algebraic closure of the field of $p$-adic
numbers. Then $\mathbb C_p$ is an algebraically closed, topologically separable, non-archimedean
complete valued field, and thus the matrix group
$\operatorname{SL}_n(\mathbb C_p)$ is a Polish group. Since $\mathbb C_p$ is algebraically closed,
of characteristic $0$ 
and of cardinality $2^{\aleph_0}$, there is a field isomorphism $\mathbb C\cong\mathbb C_p$.
The topology on $\CC_p$ is, however, totally disconnected and not locally compact.
Moreover, the topological field $\CC_p$ can be recovered from the topological group $\operatorname{SL}_n(\mathbb C_p)$
from the action of a maximal torus on a well-chosen root subgroup.
The topological closure of the prime field $\QQ$ of $\CC_p$ is the field $\QQ_p$.
It follows that for different primes $p\neq q$, the fields $\CC_p$ and $\CC_q$ are not isomorphic as 
topological fields, since then $\QQ_p\not\cong\QQ_q$.
This shows that $\operatorname{SL}_n(\mathbb{C})\cong\operatorname{SL}_n(\mathbb C_p)$ carries many non-homeomorphic 
Polish group topologies. These complex Lie groups typically have a 
simple Lie algebra which fails to be absolutely simple.

\smallskip
The field of real numbers has only one automorphism, so this construction does not carry
over. The next example, which is due to J.~Tits, shows that real algebraic groups may
nevertheless have non-continuous automorphisms.
\subsection{Rigidity fails for certain connected perfect 
real algebraic groups.}\label{Ex5}
Let $\mathbf G$ be a linear algebraic group defined over $\mathbb R$, e.g.~$\mathbf G=\operatorname{SL}_n$.
Then the group of $\RR$-rational points $G=\mathbf G(\mathbb R)$ is a real Lie group. 
Let $\mathbb R[\delta]=\mathbb{R}[x]/(x^2)$ denote the ring of dual numbers.
The tangent bundle group $TG$ on the one hand is isomorphic to 
the semidirect product 
$\operatorname{Lie}(G)\rtimes_{\operatorname{Ad}} G$, and on the other hand 
is isomorphic
to the group of $\mathbb R[\delta]$-points $\mathbf G(\mathbb R[\delta])$. 
It is observed in \cite{Tits} that
the ring $\mathbb R[\delta]$ has uncountably many non-continuous
ring automorphisms, which extend functorially to non-continuous automorphisms of 
the Lie group $\mathbf G(\mathbb R[\delta])$.
For $\mathbf{G}=\operatorname{SO}_n$ ($n\geq 3$) or $\mathbf G=\operatorname{SL}_n$ ($n\geq 2$), the
Lie group $\mathbf G(\mathbb R[\delta])$
is a connected perfect Lie group, see Section~\ref{SemidirectSection} below. These 
Lie groups $\mathbf{G}(\mathbb R[\delta])$ are
therefore not rigid within $\mathbfcal L^\sigma$ and $\mathbfcal P$.
A particular case of this phenomenon is the group 
\[G=\mathbb R^3\rtimes_\rho\operatorname{SO}(3),\]
where $\rho:\operatorname{SO}(3)\xhookrightarrow{\ \ \ }\operatorname{GL}_3(\mathbb R)$ is the standard representation.
In this case $\rho=\operatorname{Ad}$ is the adjoint representation, 
and the resulting group $G=\operatorname{SO}_3(\mathbb R[\delta])$ is not rigid.

\section{Background on topological groups and point-set topology}
\label{TopologySection}

Our convention is that all topological groups and spaces are Hausdorff.
All other topological conditions will be stated explicitly.
A \emph{path} in a space $X$ is a continuous map $f:[0,1]\longrightarrow X$.
If there exists a nonconstant path in $X$ then
we say that $X$ \emph{contains a nonconstant path.}

A set will be called \emph{countable} if its cardinality does not exceed $\aleph_0$.

The identity component of a topological group $G$ is denoted by $G^\circ$.
This is a closed normal subgroup, and $G/G^\circ$ is a totally disconnected
topological group
\cite[II.7.1 and II.7.3]{HewittRoss}. The topological group $G$ is called 
\emph{almost connected} if $G/G^\circ$ is compact.

Recall that a space is called \emph{$\sigma$-compact} if it can be written as
a countable union of compact sets.
Since locally compact $\sigma$-compact groups play a prominent role in this
article, it seems worthwhile to take a closer look at their structure. 
Theorem \ref{LsigmaStructure} below will not be used elsewhere, 
but it may well take the magic
out of this class of groups.

\begin{Thm}\label{LsigmaStructure}
For a locally compact group $G$ the following are equivalent.
\begin{enumerate}[\rm(1)]
 \item 
$G$ is $\sigma$-compact.

\item Every open subgroup $H\subseteq G$ has countable index.

\item $G$ has an almost connected open subgroup of countable index. 
\end{enumerate}

\end{Thm}
\begin{proof}
If $G$ is $\sigma$-compact and $H\subseteq G$ is open, then $G/H$ is discrete and
$\sigma$-compact, and thus countable. Hence (1) implies (2).
By \cite[Lemma~2.3.1]{MoZi}, (2) implies (3).

Going from (3) to (1) is much deeper.
Suppose that $H\subseteq G$ is an almost connected open subgroup of countable index.
We must show that $G$ is $\sigma$-compact.
By Iwasawa's Splitting Theorem, see \cite[p.~547, Theorem~1]{iwa}, \cite[Theorem A]{Gluskov} 
for the local version and 
\cite[Theorem 4.1]{hofmori},
\cite[Theorem 4.4]{HofmannKramer} for the global version,
there exists a compact subgroup $K\subseteq H$, a $1$-connected Lie group $L$
and an open continuous homomorphism $\phi:L\times K\longrightarrow H$ with discrete kernel,
and with $\phi(1,k)=k$. In particular, the group
$H'=\phi(L\times K)$ is open in $H$. Every connected locally compact group is compactly
generated and thus $\sigma$-compact. Therefore $L$ is $\sigma$-compact, and hence
$H'$ is $\sigma$-compact. The homomorphism $\pi:H\longrightarrow H/H^\circ$
annihilates the connected group $\phi(L\times\{1\})$, 
whence $\pi(H')=\pi(\phi(\{1\}\times K))$ is compact and open in $H/H^\circ$.
Thus $\pi(H')$ has finite index in $H/H^\circ$.
Since $H'\subseteq H$ is open, we have that $H^\circ\subseteq H'$ and thus
$H'$ has finite index in $H$. In particular, the $\sigma$-compact group
$H'$ has countable index in $G$. Then $G$, being a countable union
of $\sigma$-compact sets, is $\sigma$-compact.

\smallskip
An alternative but equally nontrivial proof of a more geometric nature
may be instructive.
Let $H$ be an open  almost connected
subgroup  of countable index in $G$.
We have to show that $H$ is $\sigma$-compact, for then $G$ will be $\sigma$-compact as well.
Now $H$ contains a subspace $E$ homeomorphic to 
${\mathbb R}^n$ for
some  $n\in\{0,1,2,\dots\}$ and a (maximal) compact subgroup $C$ such that
$(e,c)\longmapsto ec:E\times C\to H$ 
is a homeomorphism. For a proof that is even valid for almost connected
pro-Lie groups see~\cite[Corollary 8.5 on~p.~380]{HoMoAlmost}, and also \cite{HoMoAxioms}.
Since ${\mathbb R}^n\times C$ 
is clearly $\sigma$-compact, so is $H$.
\end{proof}

In order to prove continuity and openness of abstract group homomorphisms, we
will use Pettis' Theorem. This requires several notions related to Baire's Category 
Theorem, which we review now.

\subsection{Meager sets, Baire spaces and almost open sets}
A subset $A$ of a space $X$ is called \emph{nowhere dense} if its closure $\overline A$ has empty 
interior. A subset  $B\subseteq X$ which is contained in a countable union of nowhere dense sets is called
\emph{meager} (or of \emph{first category} in older books). 
A countable union of meager sets is again meager.
If every meager subset of $X$ has empty interior, then $X$ is called a
\emph{Baire space}. By Baire's Category Theorem, every locally compact space and every completely metrizable space is
a Baire space \cite[XI.10.1 and XIV.4.1]{Dug}.

A subset $E\subseteq X$ of a space $X$ is called \emph{almost open} if there exists an open set $U\subseteq X$
such that the symmetric difference $(E\setminus U)\cup(U\setminus E)$ is meager. Almost open sets are also called
\emph{Baire measurable}. The almost open sets form a $\sigma$-algebra \cite[Theorem 4.3]{Oxtoby},
which contains all Borel sets in~$X$.

\subsection{Polish spaces and analytic sets}
A space is called \emph{Polish} if it is second countable and metrizable by
a complete metric. 
A subset $A$ of a Polish space $X$ is called \emph{analytic} or \emph{Suslin} if 
there exists a Polish space $Z$ and an continuous map $\psi:Z\longrightarrow X$ with
$A=\psi(X)$. Analytic sets have remarkable properties; among other things, they are
almost open. On the other hand, the class of Polish spaces is way too narrow for 
our purposes. We thus propose in \ref{AlmostPolishDef} a formal definition which captures those properties
of Polish spaces and analytic sets which are important for us.
We refer to \cite{BourbakiTopology}, \cite{Kechris}, and \cite{Kur}
for more results on Polish spaces and analytic sets.
\subsection{$\mathbfcal K$-analytic sets}
Let $\mathbfcal K$ be a class of topological spaces which is closed under finite
products and under passage to closed subsets. 
In other words, if $X,Y\in\mathbfcal K$ and if $A\subseteq X$ is closed, then 
$A,X\times Y\in\mathbfcal K$.
Examples of classes satisfying these assumptions are the class $\mathbfcal P$ of Polish spaces,
the class $\mathbfcal C$ of compact spaces, and the class
$\mathbfcal L^\sigma$ of locally compact $\sigma$-compact spaces.

Given such a class $\mathbfcal K$ and a space $X$ in $\mathbfcal K$, we call a
subset $A\subseteq X$ a \emph{$\mathbfcal K$-analytic set} if there exists a
space $Z\in \mathbfcal K$ and a continuous map $\psi:Z\longrightarrow X$ with 
$\psi(Z)=A$. We let $\mathbfcal K_a$ denote the class of all pairs
$(X,A)$, where $X\in \mathbfcal K$ and $A\subseteq X$ is $\mathbfcal K$-analytic.
The following is straightforward.
\begin{Lem}\label{Pullback}
If $X\in\mathbfcal K$ and $(Y,B)\in\mathbfcal K_a$ and if $\phi:X\longrightarrow Y$
is continuous, then $(X,\phi^{-1}(B))\in\mathbfcal K_a$.
If $(X,A_1),(X,A_2)\in\mathbfcal K_a$, then $(X,A_1\cap A_2)\in\mathbfcal K_a$.
\end{Lem}
\begin{proof}
Let $Z\in\mathbfcal K$ and let $\psi:Z\longrightarrow Y$ be a continuous map with
$\psi(Z)=B$. Let $Q$ denote the pullback of the diagram 
$X\xrightarrow{\ \phi\ } Y\xleftarrow{\ \psi\ } Z$,
i.e. $Q=\{(x,z)\in X\times Z\mid \phi(x)=\psi(z)\}$. Then the projection onto
the first coordinate maps $Q$ onto $\phi^{-1}(B)$. Since $Q$ is closed in
$X\times Z$, it is contained in $\mathbfcal K$. Hence $\phi^{-1}(B)$ is
$\mathbfcal K$-analytic.

For the second claim, consider the diagonal embedding $X\longrightarrow X\times X$.
The preimage of $A_1\times A_2$ is $A_1\cap A_2$. Since $A_1\times A_2$ is
$\mathbfcal K$-analytic, the same is true by the first claim for its preimage $A_1\cap A_2$.
\end{proof}

For the following lemma recall that 
$\mathbfcal P$ is the class of Polish spaces, 
${\mathbfcal L}^\sigma$ is the class of locally compact and $\sigma$-compact  spaces, and 
$\mathbfcal C$ is the class of compact spaces, 
\begin{Lem} For these classes $\mathbfcal K$, the classes $\mathbfcal K_a$ are as follows:
\begin{enumerate}[\rm(1)]
\item  $\mathbfcal P_a$ consists of all pairs
$(X,A)$, where $X$ is Polish and $A\subseteq X$ is an analytic set (in the sense of descriptive
set theory).
\item  $\mathbfcal L^\sigma_a$ consists of all pairs
$(X,A)$, where $X$ is locally compact $\sigma$-compact and $A\subseteq X$ is 
$\sigma$-compact.
\item For the class $\mathbfcal C$ of compact spaces, $\mathbfcal C_a$ consists of all 
pairs $(X,A)$, where $X$ is compact and $A\subseteq X$ is closed.
\end{enumerate}
\end{Lem}
\begin{proof}
Claim (1) is true by definition.
For claim (2)  we note first of all that the continuous image $A$ of a $\sigma$-compact
space $Z$ is again $\sigma$-compact. Conversely, if 
$A=\bigcup_{i\geq 0}C_i$ for a countable family of compact sets $C_i\subseteq X$, then the 
topological coproduct
(the disjoint union)
$Z=\coprod_{i\geq 0} C_i$ is locally compact $\sigma$-compact and maps continuously onto $A$.
Claim (3) is  clear.
\end{proof}

Recall that a space is called
\emph{Lindel\"of} if every open covering has a countable subcovering.
Both second countable spaces and $\sigma$-compact spaces are Lindel\"of.
In particular, every Polish space and every locally compact $\sigma$-compact space
is Lindel\"of.
\begin{Def}\label{AlmostPolishDef}
We call a class $\mathbfcal K$ of spaces which is closed under finite products and
under passage to closed subsets \emph{almost Polish} if it has the following three properties.
\begin{enumerate}[\rm(1)]
 \item 
 Every member of $\mathbfcal K$ is Lindel\"of.

\item  No nonempty member of $\mathbfcal K$ is meager in itself.

\item Every $\mathbfcal K$-analytic set is almost open.
\end{enumerate}
\end{Def}

\begin{Prop}
The class $\mathbfcal P$ of Polish spaces, 
the class $\mathbfcal L^\sigma$ of
locally compact $\sigma$-compact spaces,
and the class $\mathbfcal C$ of compact spaces
are almost Polish.
\end{Prop}
\begin{proof}
We noted already that every second countable space and every $\sigma$-compact space
is Lindel\"of. Also, every complete metric space and every locally compact space is
a Baire space and hence not meager in itself
\cite[XI.10.1 and XIV.4.1]{Dug}. The fact that every analytic set in a Polish
space is almost open is proved in \cite[p.~482]{Kur}.
Hence $\mathbfcal P$ is almost Polish.
On the other hand, every $\sigma$-compact set is clearly a Borel set, and
Borel sets are almost open. 
Therefore $\mathbfcal L^\sigma$ is almost Polish.
For the class $\mathbfcal C$ we note that in particular every closed set is almost open,
whence $\mathbfcal C$ is almost Polish.
\end{proof}
For groups in an almost Polish class we have a continuity and open mapping theorem as follows.

\begin{Thm}\label{ContinuousOpenTheorem}
Let $\mathbfcal K$ be a class of almost Polish spaces. Suppose that 
$K,G$ are topological groups which belong to $\mathbfcal K$, and that 
\[K\xrightarrow{\ \phi\ } G\]
is an abstract group homomorphism.

Assume that for every identity neighborhood $U\subseteq G$ there exists an identity neighborhood $V\subseteq U$
such that $\phi^{-1}(V)$ is almost open (which is the case if $\phi^{-1}(V)$ is $\mathbfcal K$-analytic).
Then $\phi$ is continuous.

If the homomorphism $\phi$ is continuous and surjective, then $\phi$ is open.
\end{Thm}
\begin{proof}
Our proof is based on Pettis' Theorem, which says the following.
If $E$ is a subset of a topological group, and if $E$ is almost open and not meager,
then $EE^{-1}$ is an identity neighborhood \cite[Theorem 1]{Pettis}.

Now we prove the first claim.
Given an identity neighborhood $U\subseteq G$, 
we choose a smaller identity neighborhood $V\subseteq U$ such that $VV^{-1}\subseteq U$, and
such that $E=\phi^{-1}(V)$ is almost open.
Since $G$ is Lindel\"of, the closed set $\overline{\phi(K)}$ is also Lindel\"of.
Hence there exists a countable set of elements $g_i\in K$, for $i\in\mathbb N$,
such that $\phi(K)\subseteq\bigcup_{i\geq0}\phi(g_i)V$. Therefore $K=\bigcup_{i\geq 0}g_iE$.
Since $K$ is not meager in itself, $E$ cannot be meager. 
By Pettis' Theorem, $EE^{-1}$ is an identity neighborhood in $K$, with $\phi(EE^{-1})\subseteq U$. It follows
that $\phi$ is continuous at the identity, and hence continuous everywhere
\cite[III. Proposition~23]{BourbakiTopology}.

For the second claim let $W\subseteq K$ be open, let $g\in W$ and let $C\subseteq K$ be a 
closed identity neighborhood with $gCC^{-1}\subseteq W$.
Such a neighborhood $C$ exists since every topological group is regular \cite[II.4.8]{HewittRoss}.
Then $\phi(C)=D$ is $\mathbfcal K$-analytic.
Since $K$ is Lindel\"of, there exists a countable set of elements $a_i\in K$, for $i\in\mathbb N$, such that 
$K=\bigcup_{i\geq 0}a_iC$. It follows that $G=\bigcup_{i\geq 0}\phi(a_i)D$.
Since $G$ is not meager in itself,
$D$ cannot be meager. By Pettis' Theorem, $DD^{-1}$ is an identity neighborhood,
and hence $\phi(g)DD^{-1}$ is a neighborhood of $\phi(g)$ which is contained in $\phi(W)$.
Thus $\phi(W)$ is open.
\end{proof}
In order to apply Theorem \ref{ContinuousOpenTheorem}, we need conditions
which ensure that we can construct many $\mathbfcal K$-analytic sets.
The following lemma supplies such conditions.
\begin{Lem}
\label{LittleLemma}
Let $\mathbfcal K$ be a class of spaces which is closed under finite products 
and the passage to
closed subsets and let $K$ be a topological group 
in $\mathbfcal K$, with a normal 
(but not necessarily closed) $\mathbfcal{K}$-analytic subgroup
$N$ of $K$.
Put $H=K/N$ (as an abstract group) and let $\pi:K\longrightarrow H$ 
denote the natural quotient homomorphism.
Let $A,B\subseteq H$ be subsets, and assume that $\pi^{-1}(A)$ and $\pi^{-1}(B)$ 
are $\mathbfcal{K}$-analytic.
Then the following sets are also $\mathbfcal{K}$-analytic:
\begin{enumerate}[\rm(1)]
 \item the set $\pi^{-1}(AB)$,
 \item the set $\pi^{-1}(\{[a,b]\mid a\in A,\, b\in B\})$,
 \item all sets $\pi^{-1}(\{h\in H\mid [h,y]\in A\})$ for  $y\in H$, and
 \item all sets $\pi^{-1}(\operatorname{Cen}_H(X))$ for every finite set $X\subseteq H$.
\end{enumerate}
\end{Lem}
\begin{proof}
Let $s:H\longrightarrow K$ be a cross section for $\pi$, i.e. $\pi\circ s=\operatorname{id}_H$.
For every subset $Y\subseteq H$ we have $\pi^{-1}(Y)=s(Y)N$, and $s(xy)N=s(x)s(y)N$ holds
for all $x,y\in H$.
We note that for all $\mathbfcal K$-analytic subsets $P,Q\subseteq K$,
the product $PQ\subseteq K$ is again $\mathbfcal K$-analytic.

Claim (1) follows from the identity 
\[\pi^{-1}(AB)=s(AB)N=s(A)s(B)N=s(A)Ns(B)N=\pi^{-1}(A)\pi^{-1}(B).\]
The rightmost term is $\mathbfcal K$-analytic, since products of $\mathbfcal K$-analytic sets are again
$\mathbfcal K$-analytic.

For claim (2) we note that $\pi^{-1}([a,b])=s([a,b])N=[s(a),s(b)]N$,
whence \[\pi^{-1}(\{[a,b]\mid a\in A,\, b\in B\})=\{[x,y]n\mid \pi(x)\in A,\pi(y)\in B, n\in N\}.\]
This set is $\mathbfcal K$-analytic since products and inverses
of $\mathbfcal K$-analytic sets in the group $K$ are again $\mathbfcal K$-analytic.

Claim (3) follows from the identity
\[\pi^{-1}(\{h\in H\mid [h,y]\in A\})=\{g\in K\mid [g,s(y)]\in AN\}.\]
The right-hand side is $\mathbfcal K$-analytic by Lemma~\ref{Pullback}.

For (4) we note that $\pi^{-1}(\operatorname{Cen}_H(x))=\{g\in K\mid [g,s(x)]\in N\}$
by (3), and this set is $\mathbfcal K$-analytic.
Since a finite intersection of $\mathbfcal K$-analytic sets is again $\mathbfcal K$-analytic
by Lemma~\ref{Pullback}, claim (4) follows.
\end{proof}

\section{Building up good neighborhoods in Lie groups}
\label{YamabeBaireSection}

Our basic strategy for proving rigidity of Lie groups
in almost Polish classes
is to construct a neighborhood basis of the identity 
in the Lie groups which consists of sets
that arise in a purely group-theoretic way. 

Guided by ideas which go back to van der Waerden,
we proceed in two steps. In the first step
we construct one specific compact identity neighborhood.
In the second step we show that we may
`shrink' this given neighborhood in a purely group-theoretic way
to an arbitrarily  small compact identity neighborhood. This
gives us the desired neighborhood basis of the identity.
The whole method depends on the existence of nontrivial commutators,
both in the Lie algebra and in the Lie group.

\subsection{Lie groups}\label{LieGroups}
Our convention is that 
a \emph{Lie group} $G$ is a locally compact group which is a smooth real manifold,
such that multiplication and inversion are smooth maps.
A priori, no countability assumptions will be imposed on Lie groups.
We remark that a connected Lie group is automatically second countable
and $\sigma$-compact. In particular, a Lie group $G$ is $\sigma$-compact if and only 
if $G^\circ$ has countable index in $G$, or equivalently, if $G$ is second countable.
The identity component of a Lie group is open, and every closed
subgroup of a Lie group is again a Lie group. This fact will be used without 
further mention.
The Lie algebra of a Lie group $G$ is denoted by $\operatorname{Lie}(G)$.

An \emph{analytic Lie subgroup} of a Lie group $G$ is a subgroup of the form 
$H=\langle\exp(\mathfrak h)\rangle$, where $\mathfrak h\subseteq\operatorname{Lie}(G)$ is a Lie 
subalgebra.\footnote{\ Such a subgroup is also called an \emph{integral subgroup}, a \emph{virtual subgroup}
or an \emph{analytic subgroup}---the terminology varies widely.}
Such an analytic Lie subgroup is path connected, but not necessarily closed. 
However, there exists always a connected Lie group 
$L$ and an injective Lie group homomorphism $\iota:L\longrightarrow G$ with $\iota(L)=H$.
The Lie group $L$ is unique up to canonical isomorphism \cite[3.19]{Warner}.
Note that the corestriction 
$L\longrightarrow H$ of $\iota$ is continuous and
bijective, but may fail to be open. The unique Lie 
subalgebra $\mathfrak h$ of $\operatorname{Lie}(G)$ is called \emph{the Lie algebra
of $H$}.

\subsection{Spacious sets}\label{Spacious}
Let $G$ be a topological group and let $C\subseteq G$ be a subset.
We call $C$ \emph{spacious} if some product of finitely many left translates of
$CC^{-1}$ has nonempty interior. 
In the latter case there exist thus elements
$g_1,\ldots,g_r\in G$, for some $r\geq 1$, such that 
\[
D=g_1CC^{-1}g_2CC^{-1}\cdots g_rCC^{-1}\] is an
identity neighborhood. Note that the property of being spacious is inherited from $C$ 
if we pass to a larger subset.

We prove below that a nonconstant path in a simple Lie group
is always spacious.
This is basically a Baire category argument, combined with Yamabe's Theorem.
We formalize this below.

But first we recall that a \emph{poset} $(X,\le)$ is a partially ordered set,
i.e. $\leq$ is a reflexive, antisymmetric and transitive binary relation on $X$.
A \emph{chain} in a poset is a totally ordered subset. A poset is \emph{inductive} if every chain has an upper bound. 
In an inductive poset, by Zorn's Lemma,  for any element $x$ there is a maximal 
element $m$ such that $x\le m$. A poset in which every chain is
finite is trivially inductive. For example, the set of all vector subspaces of a finite
dimensional vector space is such an inductive poset under inclusion.  
A \emph{poset map} $f$ is a map between posets that preserves the partial order,
i.e. $a\leq b$ implies that $f(a)\leq f(b)$.
\subsection{The Baire--Yamabe Process}
\label{BaireYamabeProcess}%
Let $G$ be a Lie group. We denote by $\operatorname{Fin}(G)$ the poset of all finite subsets
of $G$ containing the identity, ordered by inclusion.
By $\operatorname{PCS}(G)$ we denote the poset of all path connected subgroups of $G$,
again ordered by inclusion.

We define a poset map as follows. Let $C\subseteq G$ be a fixed
subset
and let $P$ denote the path component of the identity in $CC^{-1}$. Then $P$ is 
symmetric, path connected, and~$1\in P$. If $C$ contains a nonconstant path, then 
$P$ also contains an nonconstant path.
For any $F\in\operatorname{Fin}(G)$ we put 
\[X_C(F)=\bigcup\{gPg^{-1}\mid g\in F\}\text{ and }Y_C(F)=\langle X_C(F)\rangle.\leqno{(*)}\] 
We note that the set $X_C(F)$ is path connected, symmetric, and contains the identity.
Therefore the group \[Y_C(F)=X_C(F)\cup X_C(F)X_C(F)\cup X_C(F)X_C(F)X_C(F)\cup\cdots\]
is also path connected.
The assignment 
\[Y_C:\operatorname{Fin}(G)\longrightarrow \operatorname{PCS}(G)\] is a poset map
(i.e preserves the partial order $\subseteq$).
We call the function $Y_C$ the \emph{Baire--Yamabe Process}.

Now we recall Yamabe's Theorem, see \cite{Goto} or \cite[Theorem 9.6.1]{HilgertNeeb}.
\begin{Thm}[Yamabe]\label{YamabesTheorem}
Let $H$ be a path connected subgroup of a Lie group $G$.
Then $H$ is an analytic Lie subgroup.
In particular, $H$ determines a unique subalgebra 
$\operatorname{Lie}(H)\subseteq\operatorname{Lie}(G)$. 
\end{Thm}
Yamabe's Theorem yields a poset bijection between the poset of path connected subgroups
$\operatorname{PCS}(G)$ and the poset $\operatorname{LSA}(\operatorname{Lie}(G))$ of all Lie subalgebras of 
$\operatorname{Lie}(G)$, ordered by inclusion. 
So the poset $\operatorname{PCS}(G)$ of all path connected subgroups of a Lie group is inductive
and  the same holds for the image $\im(Y_C)$  of the Baire--Yamabe Process. 
Note also that 
$P\subseteq Y_C(\{1\})\subseteq Y_C(F)$ holds
for all $F\in\operatorname{Fin}(G)$.

In the next proposition we consider an arbitrary  Lie group $G$ and specify 
the subset $C\subseteq G$ to be the image of a continuous path $f:[0,1]\longrightarrow G$ in $G$.
We note  that then $P=CC^{-1}$ is path connected. With theses assumptions we have the following result.
\begin{Prop}\label{BYP}
Let $Y_C(F)$ be a 
maximal member in the poset $\im(Y_C)\subseteq \operatorname{PCS}(G)$ of path connected subgroups of $G$
obtained by the Baire--Yamabe Process from $C=f([0,1])$. Then
$Y_C(F)$ is a normal subgroup containing $CC^{-1}$,
and, accordingly, its Lie algebra  $\operatorname{Lie}(Y_C(F))$ is a $G$-invariant ideal in $\operatorname{Lie}(G)$.
\end{Prop}
\begin{proof}
We have $P=CC^{-1}$, since $C$ is connected.
For an arbitrary element  $g\in G$ we have
 $Y_C(F)\subseteq Y_C(F\cup gF)$ and by the maximality of $Y_C(F)$
 in $\im(Y_C)$,
 equality holds. Hence by $(*)$ we have
\begin{multline*}
gX_C(F)g^{-1}= 
g\left(\bigcup\{hPh^{-1}\mid h\in F\}\right)g^{-1}\\
\subseteq \bigcup\{hPh^{-1}\mid h\in F\cup gF\}
=X_C({F\cup gF})\\
\subseteq Y_C(F\cup gF)=Y_C(F).
\end{multline*} 
Thus $gY_C(F)g^{-1}\subseteq Y_C(F)$ and therefore
 $Y_C(F)$ is a normal subgroup.
It follows that
$\operatorname{Lie}(Y_C(F))$ is an  ideal in $\operatorname{Lie}(G)$,
see \cite[Proposition~5.54(i)]{HMCompact} or \cite[Corollary~11.1.3]{HilgertNeeb}.
Since the analytic Lie subgroup
$Y_C(F)$ is normal in $G$, its Lie algebra is invariant under the adjoint 
action of $G$.
\end{proof}

The following proposition clarifies the role 
of spacious paths for the Baire--Yamabe Process.
\begin{Prop}\label{BaireLemma0}
Let $G$ be a Lie group and $f:[0,1]\longrightarrow G$  a path. Put $C=f([0,1])$.
Then the following two conditions are equivalent:
\begin{enumerate}[\rm(1)]
\item  $C$ is spacious in $G$.
\item  There is a finite set $F\in\operatorname{Fin}(G)$ 
       such that  $Y_C(F)=G^\circ$.
\end{enumerate}
\end{Prop}
\begin{proof} 
We denote the $k$-fold product of a set $A\subseteq G$ by $A^{\cdot k}$.
Set  $P=CC^{-1}$ and assume~(1). Then
 there exist  elements $g_1,\ldots, g_r\in G$ such that
\[D=g_1Pg_2\cdots g_rP\] is an identity neighborhood.
Define elements $h_1,\ldots, h_r$  recursively so that $h_1=g_1$
and $h_j=h_{j-1}g_j$, for $j=2,\ldots, r$. 
Put $F=\{1,h_1,\dots,h_r\}$.  Then $F\in \operatorname{Fin}(G)$ and
by $P=P^{-1}$ and $(*)$ we have
\begin{align*}
D\subseteq DD^{-1}&=(h_1Ph_1^{-1})(h_2Ph_2^{-1})\cdots (h_rP)(Ph_r^{-1})\cdots(h_1Ph_1^{-1})\\
&=(h_1Ph_1^{-1})\cdots (h_rPh_r^{-1})(h_rPh_r^{-1})\cdots(h_1Ph_1^{-1})
\subseteq X_C(F)^{\cdot 2r}\subseteq Y_C(F).
\end{align*}
The subgroup $Y_C(F)$ is  path connected on the one hand, 
whence $Y_C(F)\subseteq G^\circ$,
and it is open since it contains the identity neighborhood $D$.
Therefore it is an open and closed subgroup, whence $G^\circ\subseteq Y_C(F)$.
This proves (2).

Now assume (2).  Then there exists $F\in\operatorname{Fin}(G)$
such that $G^\circ=Y_C(F)=\bigcup_{n=1}^\infty X_C(F)^{\cdot n}$. Each of the sets $X_C(F)^{\cdot n}$ is compact.
By  Baire's Category Theorem~\cite[XI.10.3]{Dug} there is an $n$ such that $X_C(F)^{\cdot n}$ has nonempty interior.
Now
\begin{align*}
X_C(F)^{\cdot n}&=\Big(\bigcup_{h_1\in F}h_1Ph_1^{-1}\Big)\cdots
                              \Big(\bigcup_{h_n\in F}h_nPh_n^{-1}\Big)\\
                              &=
   \bigcup\{h_1Ph_1^{-1}h_2Ph_2^{-1}\cdots h_nPh_n^{-1}\mid h_1,\ldots,h_n\in F\}.
  \end{align*}
The Baire Category Theorem, applied once more to this finite union of compact sets, 
shows that there exist elements $h_1,h_2,\dots,h_n\in F$
such that the set
\[E=h_1Ph_1^{-1}h_2Ph_2^{-1}\cdots h_nPh_n^{-1}\]
 has a nonempty interior. Setting 
$g_1=h_1$, $g_2=h_1^{-1}h_2, \dots, g_n=h_{n-1}^{-1}h_n$ in $G$, 
we see that $Eh_n=g_1Pg_2\cdots g_nP$ 
has  nonempty interior, whence $C$ is spacious in $G$ by definition.
\end{proof}

\begin{Prop}\label{BaireLemma}
Suppose that $G$ is a Lie group whose Lie algebra is simple, 
and that \[f:[0,1]\longrightarrow G\]
is a continuous nonconstant path. Then $C=f([0,1])$ is spacious in $G$.
\end{Prop}
\begin{proof}
Put $P=CC^{-1}$ and
let $Y_C(F)$ be a maximal element resulting from the Baire--Yamabe process.
The Lie algebra $\operatorname{Lie}(G)$ is simple and $\{1\}\neq P$, whence $Y_C(F)\neq\{1\}$.
By Proposition~\ref{BYP}, the Lie algebra $\operatorname{Lie}(Y_C(F))$ is a nonzero
ideal in $\operatorname{Lie}(G)$. Hence it equals $\operatorname{Lie}(G)$, and thus $Y_C(F)=G^\circ$.
The claim follows now from Proposition~\ref{BaireLemma0}.
\end{proof}
Now we show how to shrink compact identity neighborhoods in perfect Lie groups in
an algebraic way, using commutators. This method is basically
due to van der Waerden~\cite{vdW}. The present approach follows 
closely \cite[5.59]{HMCompact}, but
avoids the BCH multiplication in Banach algebras.

We start with a lemma about Lie algebras.
\begin{Lem}\label{AnotherLemma}
Let $Z_1=[X_1,Y_1],\ldots,Z_n=[X_n,Y_n]$ be $n$ linearly 
independent commutators in a finite dimensional
real Lie algebra $\mathfrak g$. Then there exists a real number 
$r>0$ such that for all real numbers $t_1,\ldots,t_n$ with with $0<|t_1|,\ldots,|t_n|\leq r$,
the $n$ vectors
\[\widetilde Z_i(t)=\exp(\operatorname{ad}(t_iX_i))Y_i-Y_i\]
are linearly independent.
\end{Lem}
\begin{proof}
In the $n$th exterior power of the Lie algebra $\mathfrak{g}$ we have 
$Z_1\wedge\cdots\wedge Z_n\neq 0$.
From the series expansion $\exp(\operatorname{ad}(tX))=\sum_{k=0}^\infty\frac1{k!}(\operatorname{ad}(tX))^k$
we see that \[\frac1t\bigl(\exp(\operatorname{ad}(tX))Y-Y\bigr)=[X,Y]+tF(t,X,Y),\]
for some continuous function $F:\mathbb R\times\mathfrak g\times\mathfrak g\longrightarrow\mathfrak g$.
Put $Z_k(t)=\frac1t\bigl(\exp(\operatorname{ad}(tX_k))Y_k-Y_k\bigr)$, with $Z_k(0)=Z_k$.
Thus $Z_k(t)$ depends continuously on $t$.
By the continuity of the map $(t_1,\ldots,t_n)\longmapsto Z_1(t_1)\wedge\cdots\wedge Z_n(t_n)$
at $(0,\ldots,0)$, 
there exists a constant $r>0$ such that $Z_1(t_1)\wedge\cdots\wedge Z_n(t_n)\neq 0$
for all $|t_1|,\ldots,|t_n|\leq r$. Hence 
\[\widetilde Z_1(t_1)\wedge\cdots\wedge \widetilde Z_n(t_n)=t_1\cdots t_nZ_1(t_1)\wedge\cdots\wedge Z_n(t_n)\neq 0,\]  provided that 
$t_1\cdots t_n\neq 0$ and $|t_1|,\ldots,|t_n|\leq r$.
\end{proof}
The next fact we need from point-set topology is well-known and follows from the Tube Lemma
\cite[XI.2.6]{Dug}.
\begin{Lem}[Wallace]\label{Wallace}
Let $X,Y,Z$ be spaces and let $\phi: X\times Y\longrightarrow Z$ 
be a continuous map.
Let $A$ and $B$ be compact subsets of $X$ and $Y$, respectively.
Suppose that $W$ is an open set of $Z$ containing $\phi(A\times B)$.
Then there exist  neighborhoods $U$ of $A$ in $X$ and $V$ of $B$ in $Y$
such that  $\phi(U\times V)\subseteq W$.
\end{Lem}

We obtain the following variation of van der Waerden's Theorem \cite{vdW}.
\begin{Thm}[{\cite[Proposition~5.59]{HMCompact}}]\label{vdW2}
Let $G$ be an $n$-dimensional Lie
group whose Lie algebra is perfect, 
i.e. $\operatorname{Lie}(G)$ is spanned by
commutators. Then there exist $1$-parameter groups $c_1,\ldots,c_n$ in $G$ such that
that for all $t_1,\ldots,t_n$ with $0<|t_i|\leq1$ and every
identity neighborhood $U\subseteq G$,
the set 
\[
 M_{t_1,\ldots,t_n}=[c_1(t_1),U]\cdots[c_n(t_n),U]\subseteq G
\]
is an identity neighborhood.
If $U$ is compact and if $W\subseteq G$ is any identity neighborhood, 
then the $t_i>0$ can be chosen in such a way
that $M_{t_1,\ldots,t_n}\subseteq W$.
\end{Thm}  
\begin{proof}
Let $Z_1=[X_1,Y_1], \ldots, Z_n=[X_n,Y_n]$ be a basis 
of the Lie algebra  $\operatorname{Lie}(G)$.
By Lemma~\ref{AnotherLemma} there exists a number $r>0$ such that
for all real numbers $t_1,\ldots,t_n$ with $0<|t_i|\leq r$ the $n$ vectors
$\widetilde Z_{i}(t)=\exp(\operatorname{ad}(t_iX_i))Y_i-Y_i$ form a basis of $\operatorname{Lie}(G)$.
Put $c_i(t)=\exp(rtX_i)$ and $y_i(t)=\exp(rtY_i)$.

Suppose that $0<|t_i|\leq 1$.
Differentiating the smooth function $s\longmapsto [c_i(t_i),y_i(s)]$ at the time $s=0$,
we obtain the vector $\widetilde Z_i(t_i)=\exp(\operatorname{ad}(t_iX_i))Y_i-Y_i$.
From the inverse function theorem we conclude that near $(0,\ldots,0)$ the map
\[(s_1,\ldots,s_n)\longmapsto [c_1(t_1),y_1(s_1)]\cdots[c_n(t_n),y_n(s_n)]\] is a
diffeomorphism from $\mathbb R^n$ to $G$. 
Hence $M=[c_1(t_1),U]\cdots[c_n(t_n),U]$ is an identity neighborhood.
If $U$ is compact, then by Wallace's Lemma \ref{Wallace} we may
choose the $t_i$ in such a way that $M\subseteq W$.
\end{proof}
\subsection{Remark}
We note the following consequence of Theorem~\ref{vdW2}. 
In an $n$-di\-men\-sio\-nal Lie group $G$ whose Lie algebra is perfect,
every element near the identity is a product of at most $n$ commutators of elements
coming from a small identity neighborhood. In particular, the abstract
commutator subgroup of $G$ is open in $G$ 
(this follows also from Yamabe's Theorem~\ref{YamabesTheorem}).

\section{Rigidity of semisimple Lie groups}

The following result is essential 
for almost all our rigidity results concerning Lie groups.
\label{SemisimpleSection}
\begin{Thm}
\label{MainThm}
Let $\mathbfcal K$ be an almost Polish class and
suppose that $G$ is a Lie group in $\mathbfcal K$ whose Lie algebra is perfect.
Suppose also that 
\[
 1\longrightarrow N \xhookrightarrow{\ \ \ } K\xrightarrow{\ \phi\ }G\longrightarrow 1
\]
is a short exact sequence, for some abstract homomorphism $\phi$,
and that $(K,N)\in\mathbfcal K_a$.
If there exists a compact spacious subset $C\subseteq G$ whose preimage $\phi^{-1}(C)$ is 
$\mathbfcal K$-analytic, then  $\phi$ is continuous and open.
\end{Thm}
\begin{proof}
By our assumptions on $C$ we find elements $g_1,\ldots,g_r\in G$ such that 
the compact set $D=g_1CC^{-1}g_2CC^{-1}\cdots g_{r}CC^{-1}$
is a compact identity neighborhood. An iterated application of 
 Lemma \ref{LittleLemma}(1) shows that its preimage
$\phi^{-1}(D)$ is $\mathbfcal K$-analytic.
Let $W\subseteq G$ be an arbitrary identity neighborhood. By Theorem~\ref{vdW2}
we can find elements $a_1,\ldots,a_n\in G$ such that $M=[a_1,D]\cdots [a_n,D]\subseteq W$
is an identity neighborhood. Now we apply  Lemma \ref{LittleLemma}(2) 
and (1) and conclude that  its preimage $\phi^{-1}(M)$ is $\mathbfcal K$-analytic.
By Theorem \ref{ContinuousOpenTheorem}, the homomorphism $\phi$ is continuous and open.
\end{proof}
The following immediate consequence will be generalized below.
Under the additional assumption that $\phi$ is an isomorphism, $G$ is connected and that 
$\mathbfcal K$ is $\mathbfcal L^\sigma$ or $\mathbfcal P$,
such a result is proved in \cite{Pejic,Ka74,Kal82}, respectively.
\begin{Cor}
Let $G$ be a compact Lie group whose Lie algebra is semisimple.
Then $G$ is rigid within every almost Polish class $\mathbfcal K$ that contains $G$.
\end{Cor}
\begin{proof}
We apply Theorem~\ref{MainThm} to the spacious set $C=G$.
\end{proof}
The following well-known fact about Lie groups will be used several times.
\begin{Lem}\label{tfg}
Let $G$ be a Lie group. If $G/G^\circ$ is finitely generated, then 
$G$ has a finitely generated dense subgroup.
\end{Lem}
\begin{proof}
Suppose first that $G$ is connected.
We proceed by induction on the dimension $n$ of the connected Lie group $G$.
The case $n=0$ is trivial.
In general, let $H\subseteq G$ be a maximal closed connected proper subgroup in $G$.
Such subgroups exist, since $\dim(G)$ is finite, and $\dim(H)<\dim(G)$.
By the induction hypothesis, there exists a finite set $X\subseteq H$ which
generates a dense subgroup of $H$.
Let $c:\mathbb R\longrightarrow G$ be a $1$-parameter group whose image is not 
contained in $H$. Then $c(\mathbb R)\cup H$ generates a connected subgroup
$L$ of $G$. Hence the closure $\overline L$ of $L$ is a closed connected subgroup.
Thus $\overline L=G$ by the maximality of $H$.
The real numbers $1,\sqrt2\in\mathbb R$ generate additively a dense subgroup
in $\mathbb R$. Hence $\{c(1),c(\sqrt2)\}\cup X$ generates a dense subgroup in $G$.

In the general case, we choose elements $x_1,\ldots,x_m\in G^\circ$ which generate a dense
subgroup of $G^\circ$, and elements $y_1,\ldots,y_n\in G$ whose cosets
$y_1G^\circ,\ldots,y_nG^\circ$ generate $G/G^\circ$.
Then the finite set $\{x_1,\ldots,x_m,y_1,\ldots,y_n\}$ generates a dense subgroup of $G$.
\end{proof}
The following observation will be used below. If a Lie group $G$ acts continuously on a space
$X$, and if the stabilizer $G_x$ of the point $x\in X$ is not open
(i.e.~if $G^\circ\not\subseteq G_x$), 
then the orbit $G(x)\subseteq X$ contains a nonconstant path.
For we may choose then a $1$-parameter group $c:\mathbb R\longrightarrow G$
which is not contained in $G_x$. Then $c([0,1])$ is not contained in $G_x$,
and thus $t\longmapsto c(t)(x)$ is a nonconstant path in the orbit $G(x)$.

We recall from the introduction that a simple real Lie algebra $\mathfrak g$ 
is called \emph{absolutely simple} if its
complexification $\mathfrak g\otimes_{\mathbb R}\mathbb C$ is a simple complex Lie algebra.
We extract the following technical result from \cite{KramerUnique}.
\begin{Prop}\label{AbsolutelySimpleProp}
Let $G$ be a connected Lie group whose Lie algebra is absolutely simple.
Assume also that the center of $G$ is finite. Then there exists an element
$h\in G$ and a finite subset $X\subseteq G$ such that the set 
$C=\{ghg^{-1}\mid g\in \operatorname{Cen}_G(X)\}$ is compact and contains 
a nonconstant path.
In particular, $C$ is spacious in $G$.
\end{Prop}
\begin{proof}
If $G$ is compact, we put $X=\{1\}$ and we choose $h\in G\setminus\operatorname{Cen}(G)$.
Then the conjugacy class $C=\{ghg^{-1}\mid g\in G\}$ is compact and,
since $h$ is not central,
contains a nonconstant path by the observation recorded above.
By Proposition~\ref{BaireLemma}, the set $C$ is spacious in $G$.

Suppose now that $G$ is not compact. Then $G$ has positive real rank and,
 in particular,
there exist nontrivial parabolic subgroups in $G$.
We refer to \cite[pp. 2627--2628]{KramerUnique} for the following facts.
In \emph{loc.cit.} it is assumed that $G$ is centerless, but the reasoning
remains valid in the presence of a finite center, as we explain now.

Let $P\subseteq G$ be a next-to-minimal parabolic, i.e.~up to conjugation,
there is exactly one parabolic subgroup of $G$ properly contained in $P$.
Let $H\subseteq P$ be a reductive
Levi subgroup, so that $P=HU$, where $U$ is the unipotent radical of $P$. 
Then $H$ is a reductive group of real rank $1$
(because $P$ was next-to-minimal).
It is shown in \cite[p.~73]{WarnerHarmonic}
\footnote{\ Where $H=\mathsf L_\Theta$ in the notation of \emph{loc.cit.}}
that the group $H$ can be written as the $G$-centralizer of a closed connected abelian subgroup 
$S\subseteq G$.
Since $\operatorname{Lie}(G)$ is absolutely simple, the parabolic $P$ can be
chosen in such a way that $\mathfrak{sl}_2\mathbb C$ is not a direct factor in 
$\operatorname{Lie}(H)$, see \cite[Lemma~10]{KramerUnique}.
It is shown in \emph{loc.cit.} p.~2627 that then there exists a closed connected abelian subgroup
$T\subseteq G$ and an element $h\in G$ such that the conjugates of $h$ under the group
$L=\operatorname{Cen}_H(T)$ form a compact set containing a 
nonconstant path.
Now \[L=H\cap\operatorname{Cen}_{G}(T)=\operatorname{Cen}_{G}(S)\cap\operatorname{Cen}_{G}(T)=
\operatorname{Cen}_{G}\big(\overline{\langle S\cup T\rangle}\big).\]
By Lemma \ref{tfg} we may choose a finite set $X\subseteq \overline{\langle S\cup T\rangle}$ 
which generates a dense subgroup of $\overline{\langle S\cup T\rangle}$.
Thus $L=\operatorname{Cen}_{G}(X)$, and $C=\{ghg^{-1}\mid g\in L\}$ is compact and 
contains a nonconstant path, as required.
By Proposition~\ref{BaireLemma}, the set $C$ is spacious in $G$.
\end{proof}
If we combine Proposition~\ref{AbsolutelySimpleProp} with Theorem~\ref{MainThm},
we obtain immediately a rigidity result for Lie groups satisfying the assumptions of
Proposition~\ref{AbsolutelySimpleProp}. However, we can do better.
First we extend Proposition~\ref{AbsolutelySimpleProp} to the semisimple case.
\begin{Prop}\label{SemisimpleProp}
Let $G$ be a Lie group. Suppose that $\operatorname{Cen}(G^\circ)$ is finite and 
that the Lie algebra $\operatorname{Lie}(G)$ is a direct sum of absolutely simple ideals.
Then there exists an element $h$ and a finite subset $X$ in $G^\circ$ such that
$C=\{ghg^{-1}\mid g\in\operatorname{Cen}_G(X)\}$ is compact, spacious, and contains a 
nonconstant path.
\end{Prop}
\begin{proof}
We first consider the case where $G$ is connected.
Then $G$ is a connected semisimple Lie group with finite center.
We proceed by induction on the number $r$ of 
simple ideals in the Lie algebra of $G$.
The case $r=1$ is taken care of by 
Proposition~\ref{AbsolutelySimpleProp}.
For $r\geq 2$ we decompose the Lie algebra of $G$ into a simple ideal and a complementary
semisimple ideal. Accordingly, $G$ is a central product $G=G_1G_2$ of two closed connected commuting
subgroups $G_1,G_2$, with $\operatorname{Lie}(G_1)$ absolutely simple and where $G_2$ has a semisimple Lie
algebra which is a sum of $r-1$ absolutely simple ideals.
By the induction hypothesis we find elements $h_i\in G_i$ and finite subsets $X_i\subseteq G_i$
such that the sets $C_i=\{gh_ig^{-1}\mid g\in \operatorname{Cen}_{G_i}(X_i)\}$ 
contain nonconstant paths and are compact and spacious in $G_i$,
for $i=1,2$. Since $[G_1,G_2]=1$ and $G=G_1G_2$, we have
$\operatorname{Cen}_G(X_i)=G_{2-i}\operatorname{Cen}_{G_i}(X_i)$ and 
$C_i=\{gh_ig^{-1}\mid g\in \operatorname{Cen}_{G}(X_i)\}$.
Put $h=h_1h_2$ and $X=X_1\cup X_2$ and 
$C=\{ghg^{-1}\mid g\in\operatorname{Cen}_G(X)\}$.
Then 
\begin{align*}
 C&=\{(g_1g_2)(h_1h_2)(g_1g_2)^{-1}\mid g_i\in G_i\text{ for $i=1,2$ and } g_1g_2\in\operatorname{Cen}_G(X)\}\\
 &=\{(g_1g_2)(h_1h_2)(g_1g_2)^{-1}\mid g_i\in \operatorname{Cen}_{G_i}(X_i)\text{ for $i=1,2$}\}\\
 &=\{(g_1h_1g_1^{-1})(g_2h_2g_2^{-1})\mid g_i\in \operatorname{Cen}_{G_i}(X_i)\text{ for $i=1,2$}\}\\
 &=C_1C_2
\end{align*}
is compact and contains a nonconstant path.
There exist elements $g_{i,1},\ldots,g_{i,s}\in G_i$ such that 
$D_i=g_{i,1}C_iC_i^{-1}g_{i,2}C_iC_i^{-1}\cdots g_{i,s}C_iC_i^{-1}$ is a compact 
identity neighborhood in $G_i$, for $i=1,2$. 
We can safely assume that the number $s$ of elements is in  both cases the same,
because the product of an open set and an arbitrary set in a topological group is always open.
Consider the natural surjective homomorphism $G_1\times G_2\longrightarrow G$.
This map is open and maps $D_1\times D_2$ onto $D_1D_2$, whence $D_1D_2$ is an 
identity neighborhood in $G$. Put $g_j=g_{1,j}g_{2,j}$, for $j=1,\ldots,s$. Then 
\[
 D_1D_2=g_1CC^{-1}g_2CC^{-1}\cdots g_sCC^{-1}.
\]
Thus $C$ is spacious in $G$.

It remains to consider the case where $G$ is not connected.
Then $G^\circ$ is a connected semisimple Lie group with finite center, as above.
We  choose $h$ and $X$ in $G^\circ$ as before and we put
\[
C_0=\{ghg^{-1}\mid g\in\operatorname{Cen}_{G^\circ}(X)\}\subseteq 
\{ghg^{-1}\mid g\in\operatorname{Cen}_{G}(X)\}=C.
\]
The set $C_0$ is thus compact, spacious, and contains a nonconstant path by the argument given above.
It remains to show that $C$ is compact.
Consider the Lie group $\operatorname{Aut}_\RR(\operatorname{Lie}(G))$
of all $\mathbb R$-linear automorphisms of the Lie algebra $\operatorname{Lie}(G)$.
The group $G$ acts  on $\operatorname{Lie}(G)$ via the adjoint representation
\[
 G\xrightarrow{\ \operatorname{Ad}\ }\operatorname{Aut}_\RR(\operatorname{Lie}(G)).
\]
Put $H=\operatorname{Ad}(G)$. We note that $H$ acts faithfully on $G^\circ$ by conjugation
because $G^\circ$ is generated by $\exp(\operatorname{Lie}(G))$.
Under the homomorphism $\operatorname{Ad}$, the connected group $G^\circ$ maps onto
the identity component $\operatorname{Aut}_\RR(\operatorname{Lie}(G))^\circ$, because $\operatorname{Lie}(G)$
is semisimple and thus 
$\operatorname{Lie}(\operatorname{Aut}_\RR(\operatorname{Lie}(G)))=\operatorname{Lie}(G)$.
It follows that
$\operatorname{Aut}_\RR(\operatorname{Lie}(G))^\circ=\operatorname{Ad}(G^\circ)=H^\circ$.
The quotient $\operatorname{Aut}_\RR(\operatorname{Lie}(G))/\operatorname{Aut}_\RR(\operatorname{Lie}(G))^\circ$
is finite, see \cite[Corollary~2]{Murakami}, \cite[Proposition~13.1.5]{HilgertNeeb}, 
or \cite{HG} for a stronger structural result.
Hence $[H:H^\circ]$ is also finite.
Let $H_X$ denote the point-wise stabilizer of $X$ under the $H$-action
on $G^\circ$.
Then $(H^\circ)_X=H_X\cap H^\circ$ has finite index in $H_X$.
Hence $[\operatorname{Cen}_{G}(X):\operatorname{Cen}_{G^\circ}(X)]$ is finite
as well, and we may put 
\[\operatorname{Cen}_{G}(X)=a_1\operatorname{Cen}_{G^\circ}(X)\cup \cdots\cup a_k\operatorname{Cen}_{G^\circ}(X),\]
for elements $a_1,\ldots,a_k\in G$.
Thus
\[
 C=a_1C_0a_1^{-1}\cup\cdots\cup a_kC_0a_k^{-1}
\]
is compact.
\end{proof}
The following result is Theorem A in the Introduction.
\begin{Thm}\label{SemisimpleThm}
Let $G$ be a Lie group. Suppose that $\operatorname{Cen}(G^\circ)$ is finite and that 
the Lie algebra $\operatorname{Lie}(G)$ is a direct sum of absolutely simple ideals.
Let $\mathbfcal K$ be an almost Polish class containing $G$.
Then $G$ is rigid within $\mathbfcal K$.
\end{Thm}
\begin{proof}
Let $h$ and $X$ be as in Proposition~\ref{SemisimpleProp}
and put $C=\{ghg^{-1}\mid g\in\operatorname{Cen}_G(X)\}$.
Then $C$ is compact and spacious by \emph{loc.cit.}
Suppose that 
\[
 1\longrightarrow N \xhookrightarrow{\ \ \ } K\xrightarrow{\ \phi\ }G\longrightarrow 1
\]
is a short exact sequence, for some abstract homomorphism $\phi$,
and that $(K,N)\in\mathbfcal K_a$.
The preimage of $C$ is  $\mathbfcal K$-analytic by Lemma~\ref{LittleLemma}(4).
Hence $G$ is rigid by Theorem~\ref{MainThm}.
\end{proof}

\section{The case of semidirect products of Lie groups}
\label{SemidirectSection}

In this section,  
all
Lie group representations will be assumed to be finite dimensional and continuous,
unless stated otherwise.

\subsection{Construction}
Given a real representation
$\rho: H\longrightarrow\operatorname{GL}(V)$ of a Lie group $H$ we may form 
the semidirect product \[G=V\rtimes_\rho H.\] 
The underlying manifold is
$V\times H$, and the multiplication is given by 
\[(u,a)(v,b)=(u+av,ab),\leqno{(**)}\] where
we write $\rho(a)(v)=av$ for short. 
The neutral element is $(0,1)$ and the inverse of
$(u,a)$ is \[(u,a)^{-1}=(-a^{-1}u,a^{-1}).\] Then clearly $G$ is a Lie group.
We may identify $H$ and $V$ with closed subgroups of $G$, and
$V$ is normal in $G$.

As a vector space, the Lie algebra of $G$,
which is the tangent space of the manifold $V\times H$ at the point $(0,1)$,
is given by
$\operatorname{Lie}(G)=V\oplus\operatorname{Lie}(H)$.
From this decomposition and the 
representations $\operatorname{Ad}$ and $\rho$ of $H$ on $\operatorname{Lie}(H)$ and $V$,
the Lie bracket can be worked out as
\[[(u,X),(v,Y)]=(Xv-Yu,[X,Y]),\]
where $u,v\in\operatorname{Lie}(V)=V$ and $X,Y\in\operatorname{Lie}(H)$, and 
where we put $Xv=\operatorname{Lie}(\rho)(X)(v)$ for short.
We refer to \cite[Chapter~V.3]{HiHoLa} for details.

Now we consider Lie subalgebras of $\operatorname{Lie}(G)$ containing $\operatorname{Lie}(H)$.
As an $H^\circ$-module under the adjoint action, the vector space $\operatorname{Lie}(G)$
decomposes as a direct sum of $H^\circ$-modules as
\[\operatorname{Lie}(G)=V\oplus\operatorname{Lie}(H).\]
Suppose that $\mathfrak m\subseteq \operatorname{Lie}(G)$ is a Lie subalgebra
containing $\operatorname{Lie}(H)$. Since $\mathfrak m$ is then an $H^\circ$-module,
we conclude that \[\mathfrak m=(\mathfrak m\cap V)\oplus\operatorname{Lie}(H)\]
as an $H^\circ$-module. This observation has the following consequences.
We use the notation set up so far.
\begin{Lem}\label{LieSemidirectLemma}
Let $\rho:H\longrightarrow\operatorname{GL}(V)$ be a real representation
of a Lie group $H$
and put $G=V\rtimes_\rho H$.
\begin{enumerate}[\rm(1)]
\item If the subrepresentation $H^\circ\longrightarrow\operatorname{GL}(V)$ is nontrivial, then
$\operatorname{Lie}(H)$ is not an ideal in $\operatorname{Lie}(G)$.
\item If the subrepresentation $H^\circ\longrightarrow\operatorname{GL}(V)$ is 
irreducible, then $\operatorname{Lie}(H)$ is a maximal proper subalgebra in $\operatorname{Lie}(G)$.
\item If $\operatorname{Lie}(H)$ is perfect and if the subrepresentation 
$H^\circ\longrightarrow\operatorname{GL}(V)$ is nontrivial
and irreducible, then $\operatorname{Lie}(G)$ is perfect.
\end{enumerate}
\end{Lem}
\begin{proof}
If the connected Lie group $H^\circ$ acts nontrivially on $V$ via $\rho$, then its Lie algebra
also acts nontrivially on $V$ 
via $\operatorname{Lie}(\rho)$. Hence there exists $X\in\operatorname{Lie}(H)$ and $v\in V$
with $Xv\neq 0$. Then $[(0,X),(v,0)]=(Xv,0)\not\in\operatorname{Lie}(H)$, hence $\operatorname{Lie}(H)$ is not an ideal.
This proves (1).

For (2), suppose that $\mathfrak m$ is a subalgebra of $\operatorname{Lie}(G)$ containing $\operatorname{Lie}(H)$
properly. Then $V\cap\mathfrak m\neq\{0\}$ as we noted above. 
Since $V$ is by assumption an irreducible $H^\circ$-module,
$V\cap\mathfrak m=V$ and hence $\mathfrak m=\operatorname{Lie}(G)$.

If $\operatorname{Lie}(H)$ is perfect, then 
$\operatorname{Lie}(H)=[\operatorname{Lie}(H),\operatorname{Lie}(H)]\subseteq[\operatorname{Lie}(G),\operatorname{Lie}(G)]$.
By (1), the Lie algebra $\operatorname{Lie}(H)$ is not an ideal in $\operatorname{Lie}(G)$, 
whence $\operatorname{Lie}(H)\neq[\operatorname{Lie}(G),\operatorname{Lie}(G)]$.
By (2), this implies that $[\operatorname{Lie}(G),\operatorname{Lie}(G)]=\operatorname{Lie}(G)$.
\end{proof}
\begin{Lem}\label{SemiLemma}
Suppose that $\rho:H\longrightarrow\operatorname{GL}(V)$ is a real representation of a Lie group $H$, 
that $\operatorname{Lie}(H)$ is perfect and 
that the subrepresentation $H^\circ\longrightarrow\operatorname{Lie}(G)$ is nontrivial and irreducible. 
Let $f:[0,1]\longrightarrow H$ be a path and put $C=f([0,1])$.
If $C$ is spacious in $H$, then $C$ is also spacious in $G$.
\end{Lem}
\begin{proof}
We consider first the Baire--Yamabe Process $Y_C$ in $H$.
By Lemma~\ref{BaireLemma0} there is a finite set $E\in \operatorname{Fin}(H)$ with
\[Y_C(E)=H^\circ.\] 
%
Now we consider the Baire--Yamabe Process $Y_C$ in $G$. 
Let $Y_C(F)\subseteq G$ be a maximal element in $\operatorname{im}(Y_C)$.
Thus $Y_C(F)=Y_C(E\cup F)\supseteq Y_C(E)$. 
Moreover, $\operatorname{Lie}(Y_C(F))$ is by Proposition~\ref{BYP} an ideal in $\operatorname{Lie}(G)$. 
Therefore $\operatorname{Lie}(Y_C(F))=\operatorname{Lie}(G)$ by Lemma~\ref{LieSemidirectLemma}(3),
and thus $Y_C(F)=G^\circ$.
Again by Lemma~\ref{BaireLemma0}, the compact set $C$ is spacious in $G$.
\end{proof}

For the next theorem we first recall Schur's Lemma.
If $M\subseteq \operatorname{GL}(V)$ is a subgroup acting irreducibly on the finite dimensional
real vector space $V$,
then the ring $\mathbb D=\operatorname{End}_M(V)$ of all endomorphisms of $V$ that
commute with the $M$-action is a finite dimensional division ring over $\mathbb R$, and hence
\[\mathbb D\in\{\mathbb R,\mathbb C,\mathbb H\},\] where $\mathbb H$ denotes the 
division ring of real quaternions. The vector space $V$ then becomes a right
$\mathbb D$-module, and $M$ acts as a group of $\mathbb D$-linear endomorphisms.

Suppose that $V$ is a finite dimensional right $\mathbb D$-module, for
$\mathbb D\in\{\mathbb R,\mathbb C,\mathbb H\}$.
Let $|\cdot|$ denote the standard norm on $\mathbb D$ and let 
\[U=\{u\in \mathbb D^*\mid |u|=1\}\] denote the compact group of norm 1
elements in $\mathbb D^*$.
Every element $a\in\mathbb D$ has a (unique) presentation in polar coordinates as $a=ut$, with
$u\in U$ and $t=|a|$. Put
\[
\operatorname{SL}_{\mathbb D}(V)=[\operatorname{GL}_{\mathbb D}(V),\operatorname{GL}_{\mathbb D}(V)].
\]
Then there is a short exact sequence of linear Lie groups
\[
1\longrightarrow\operatorname{SL}_{\mathbb D}(V)\xhookrightarrow{\ \ \ \ }
\operatorname{GL}_{\mathbb D}(V)\xrightarrow{\ \det_{\mathbb D}\ }K_1(\mathbb D)\longrightarrow 1.
\]
For the commutative fields $\mathbb D=\mathbb R$ and $\mathbb D=\mathbb C$, the range of
$\det_{\mathbb D}$ is $K_1(\mathbb D)=\mathbb D^*$ and $\det_{\mathbb D}$ is the usual determinant. 
In the quaternionic case $K_1(\mathbb H)=\mathbb R_{>0}$ is the multiplicative groups of
the positive reals
and $\det_{\mathbb{H}}$ is the Dieudonn\'e determinant,
which is given by 
\[\det\nolimits_{\mathbb H}=|\det\nolimits_{\mathbb C}|,\]
where $V$ is viewed as as right $\mathbb C$-module.

If $V$ is a finite dimensional right $\mathbb D$-module, we may view $\mathbb D^*$ as a closed subgroup
of $\operatorname{GL}_{\mathbb R}(V)$, by identifying $a\in \mathbb D^*$ with the linear
map $v\longmapsto va$.
Then the product
\[\operatorname{SL}_{\mathbb D}(V)\,\mathbb D^*=\operatorname{GL}_{\mathbb D}(V)\,\mathbb D^*\]
is a subgroup of $\operatorname{GL}_{\mathbb R}(V)$, which is closed by the following lemma.
\begin{Lem}
\label{ProductClosed}
Let $V$ be a finite dimensional right $\mathbb D$-module, for 
$\mathbb D\in\{\mathbb R,\mathbb C,\mathbb H\}$.
Let $M\subseteq \operatorname{SL}_{\mathbb D}(V)$ and  
$L\subseteq \mathbb D^*$ be closed subgroups. Then $ML\subseteq\operatorname{GL}_{\mathbb R}(V)$
is a closed subgroup.
\end{Lem}
\begin{proof}
We use the notation we have set up above.
Put $r=\dim_{\mathbb D}(V)$ and $Z=\operatorname{GL}_{\mathbb D}(V)\cap U$.
Then $|\det_{\mathbb D}(z)|=1$ for all $z\in Z$.
Hence the continuous homomorphism \[\operatorname{GL}_{\mathbb D}(V)\times U\longrightarrow\mathbb{R}_{>0}\]
that maps $(h,u)$ to $|\det\nolimits_{\mathbb D}(h)|$ factors over
the projection map \[\operatorname{GL}_{\mathbb D}(V)\times U\longrightarrow
\operatorname{GL}_{\mathbb D}(V)U=\operatorname{GL}_{\mathbb D}(V)\,\mathbb{D}^*.\]
Hence we have a well-defined continuous homomorphism 
\[\nu:\operatorname{GL}_{\mathbb D}(V)U\longrightarrow\mathbb R_{>0}\] 
given by
\[
\nu(hu)=|\det\nolimits_{\mathbb D}(h)|,
\]
for $h\in\operatorname{GL}_{\mathbb D}(V)$ and $u\in U$.
We note also that $\operatorname{GL}_{\mathbb D}(V)U$ is closed in $\operatorname{GL}_\RR(V)$
because $\operatorname{GL}_{\mathbb D}(V)$ is closed in $\operatorname{GL}_\RR(V)$ and $U$ is compact.

Now let $(g_k)_{k\in\mathbb N}$ be a sequence in $ML$ converging to an element 
$g\in\operatorname{GL}_{\mathbb D}(V)U$. We express every $g_k$ as a product
$g_k=s_kt_ku_k$, with $s_k\in M$, $t_k\in\mathbb R_{>0}$ and $u_k\in U$.
Then $\nu(g)=\lim_k\nu(g_k)=\lim_k|t_k|^m$. Passing to a subsequence, we may 
therefore assume that
the bounded sequences $t_k$ and $u_k$ converge to elements $t$ and $u$, respectively.
Thus $\lim_kt_ku_k=tu\in L$.
Hence \[s=\lim\nolimits_ks_k=\lim\nolimits_kg_ku_k^{-1}t_k^{-1}\] exists and is equal to $gu^{-1}t^{-1}$.
Since $M$ was assumed to be closed, $s\in M$ and hence $g=stu\in ML$.
\end{proof}
\begin{Thm}\label{SemidirectTheorem}
Let $V$ be a finite dimensional real vector space and
let $M\subseteq\operatorname{GL}(V)$ be a closed  subgroup, with $M/M^\circ$ 
finite.
Assume that $\operatorname{Lie}(M)$ is a direct 
sum of absolutely simple ideals and that $M^\circ$ acts irreducibly on $V$.
Let $L\subseteq \operatorname{End}_M(V)^*=\mathbb D^*$ be a closed nontrivial subgroup.
Let $\rho$ denote the representation of $ML$ on $V$.
Then $ML\subseteq\operatorname{GL}_{\mathbb R}(V)$ is closed and 
$G=V\rtimes_\rho ML$ is rigid within every almost Polish class $\mathbfcal K$
containing $G$.
\end{Thm}
\begin{proof}
We view $V$ as a right $\mathbb D$-module, for $\mathbb D=\operatorname{End}_M(V)$.
The Lie algebra of $M$ is perfect. Therefore $M^\circ$ is contained 
in $\operatorname{SL}_{\mathbb D}(V)$. By Lemma~\ref{ProductClosed},
the group $M^\circ L$ is closed in $\operatorname{GL}_{\mathbb R}(V)$. 
Since $M/M^\circ$ is finite, $ML$ is also closed in $\operatorname{GL}_{\mathbb R}(V)$.

We write the elements of $G$ as pairs $(u,m\ell)$, with $u\in V$, $m\in M$ and $\ell\in L$.
Because of the presence of $L$, the Lie algebra of $G$ need not be perfect, so additional work is 
required.

In any case, $\operatorname{Lie}(V\rtimes_\rho M)$ is perfect by Lemma~\ref{LieSemidirectLemma}(3).
Our first aim is to construct a good neighborhood basis of the identity of
this subgroup of $G$.
Let $z\in L$ be a nontrivial element.
Since $z$ is not the identity in the 
division ring $\mathbb D$, it fixes no nonzero vector in $V$.
Then 
\[(u,m\ell)(0,z)=(u,m\ell z)\text{ and }(0,z)(u,m\ell)=(zu,zm\ell )=(zu,mz\ell),\] whence
\[\operatorname{Cen}_G(z)=M\cdot\operatorname{Cen}_L(z).\] Since $M^\circ$ is linear and semisimple, 
$\operatorname{Cen}(M^\circ)$ is finite,
see \cite[38.5.(3)]{FdV} or \cite[Corollary~13.2.6]{HilgertNeeb}.
By Proposition~\ref{SemisimpleProp} we can find an element $h$ and a finite set $X$ in
$M$ such that \[C=\{khk^{-1}\mid k\in\operatorname{Cen}_M(X)\}\] is compact and 
spacious in $M$. Since $L$ commutes with $M$, we may rewrite this set
as
\[
 C=\{ghg^{-1}\mid g\in\operatorname{Cen}_G(X\cup\{z\})\}.
\]
By Lemma~\ref{SemiLemma}, the set $C$ is also spacious in $V\rtimes_\rho M$.
Hence there exist elements ${g_1,\ldots,g_r\in V\rtimes_\rho M}$ such that
$D=g_1CC^{-1}\cdots g_rCC^{-1}$ is a compact identity neighborhood in
$V\rtimes_\rho M$.
By Lemma~\ref{LieSemidirectLemma}(3) and Theorem~\ref{vdW2}, there exist $1$-parameter 
groups $c_1,\ldots,c_m$ in $V\rtimes_\rho M$ such that the sets 
\[M_{1_1,\ldots,t_m}=[c_1(t_1),D]\cdots[c_m(t_m),D],\]
for $0<t_i\leq 1$ and $i=1,\ldots,m$, form a neighborhood basis of the identity in
$V\rtimes_\rho M$.

Now we construct a good neighborhood basis of the identity in $L$,
using the sets $M_{t_1,\ldots,t_m}\subseteq V\rtimes_\rho M$.
Let $X'\subseteq M$ be a finite set generating a dense subgroup. Since
$M/M^\circ$ is finite, such a set exists by Lemma~\ref{tfg}.
We have \[L=\operatorname{Cen}_G(M)=\operatorname{Cen}_G(X'),\] as is easily checked
from the multiplication rule $(**)$.
Recall that $|\cdot|$ denotes
the standard norm on the real division algebra $\mathbb D$.
We fix also a norm $||\cdot||$ on the right $\mathbb D$-module $V$, and a vector
$v_0\in V$ with $||v_0||=1$.
We put 
\begin{align*}
N_{t_1,\ldots,t_m}&=\{\ell\in L\mid [(v_0,1),(0,\ell)]\in M_{t_1,\ldots,t_m}\}\\
&=\{\ell \in \operatorname{Cen}_G(X')\mid [(v_0,1),(0,\ell)]\in M_{t_1,\ldots,t_m}\}.
\end{align*}
Since
\[
 [(v_0,1),(0,\ell)]=(v_0,1)(0,\ell)(-v_0,1)(0,\ell^{-1})=(v_0,\ell)(-v_0,\ell^{-1})=(v_0(1-\ell),1)
\]
and \[ ||v_0(\ell-1)||=|\ell-1|,\] 
we conclude that this family of sets is indeed a neighborhood basis of the 
identity in $L$, where $0<t_i\leq 1$ and $i=1,\ldots,m$.

We now combine the various pieces of this proof. The natural Lie group homomorphism
\[
 (V\rtimes_\rho M)\rtimes L\longrightarrow V\rtimes_\rho ML=G
\]
is continuous and open.
Given an identity neighborhood $W\subseteq G$, there exists therefore
an identity neighborhood $M_{t_1,\ldots,t_m}$ in $V\rtimes_\rho M$
and another identity neighborhood $N_{t_1',\ldots,t_m'}$ in $L$
such that the product $M_{t_1,\ldots,t_m}N_{t_1',\ldots,t_m'}\subseteq W$
is an identity neighborhood in $G$.

Suppose that $K$ is another group in the class $\mathbfcal K$
and that $\phi:K\longrightarrow G$ is an abstract surjective homomorphism
whose kernel $N$ is $\mathbfcal K$-analytic.
Then both $\phi^{-1}(M_{t_1,\ldots,t_m})$ and 
$\phi^{-1}(N_{t_1',\ldots,t_m'})$ are $\mathbfcal K$-analytic 
by Lemma~\ref{LittleLemma},
and so is
$\phi^{-1}(M_{t_1,\ldots,t_m}N_{t_1',\ldots,t_m'})$.
Hence $\phi$ is continuous and open by Theorem~\ref{ContinuousOpenTheorem}.
\end{proof}
The next result generalizes some of the results in \cite{Kal16}.
\begin{Thm}\label{ClassicalGroups}
The following Lie groups are rigid in every almost Polish class $\mathbfcal K$ containing them:
\begin{align*}
 \mathbb{R}^n\rtimes_\rho\operatorname{O}(n)&\qquad n\geq 3,\\
 \mathbb{R}^{n}\rtimes_\rho\operatorname{SO}(n)&\qquad n\geq 4,\\
 \mathbb{C}^n\rtimes_\rho\operatorname{SU}(n)&\qquad n\geq 2,\\
 \mathbb{H}^n\rtimes_\rho\operatorname{Sp}(n)&\qquad n\geq 1,\\
 \mathbb{R}^{2n}\rtimes_\rho\operatorname{Sp}_{2n}(\mathbb R)&\qquad n\geq 1,\\
 \mathbb{R}^{n}\rtimes_\rho\operatorname{SL}_{n}(\mathbb R)&\qquad n\geq 2,\\
 \mathbb{R}^{n}\rtimes_\rho\operatorname{GL}_{n}(\mathbb R)&\qquad n\geq 2.
\end{align*}
In all cases, $\rho$ denotes the natural representation.
\end{Thm}
Here $\operatorname{Sp}(n)$ denotes the quaternion unitary group acting on $\mathbb H^n$, and
$\operatorname{Sp}_{2n}(\mathbb R)$ denotes the group that leaves the standard symplectic
form on $\mathbb R^{2n}$ invariant.
This list of applications of Theorem~\ref{SemidirectTheorem} to classical groups is by no means complete.

The reader should keep in mind, though, that 
\[
\mathbb R^3\rtimes_{\rho}\operatorname{SO}(3)=
\operatorname{Lie}(\operatorname{SO}(3))\rtimes_{\operatorname{Ad}}\operatorname{SO}(3)=\operatorname{SO}_3(\mathbb R[\delta])
\] 
is not rigid by Example~\ref{Ex5}. Likewise, the Lie group $\operatorname{O}_3(\mathbb R[\delta])$ is not rigid. 
Note, however,
that for the natural representation $\rho:\operatorname{O}(3)\xhookrightarrow{\ \ \ }\operatorname{GL}_3(\mathbb R)$
we have
\[
\mathbb R^3\rtimes_\rho\operatorname O(3)\not\cong 
\operatorname{Lie}(\operatorname{O}(3))
\rtimes_{\operatorname{Ad}}\operatorname{O}(3)=\operatorname{O}_3(\mathbb R[\delta]),
\]
because $-1\in\operatorname{O}(3)$ acts via $\rho$ nontrivially on $\mathbb R^3$, whereas $-1$ acts trivially on 
$\operatorname{Lie}(\operatorname{O}(3))$ via $\operatorname{Ad}$.

\begin{proof}[Proof of Theorem~\ref{ClassicalGroups}]
The assumptions of Theorem~\ref{SemidirectTheorem} are satisfied with 
$L=\{\mathrm{id}_V,-\mathrm{id}_V\}$ for the cases $\mathbb{R}^n\rtimes_\rho\operatorname{O}(n)$ with
$n\geq 3$ and for $\mathbb{R}^{2m}\rtimes_\rho\operatorname{SO}(2m)$ with $m\geq 2$.
The groups $\operatorname{SU}(n)$ for $n\geq 2$, 
$\operatorname{Sp}(n)$ for $n\geq 1$, $\operatorname{Sp}_{2n}(\RR)$ for $n\geq 1$, and
$\operatorname{SL}_{2m}(\RR)$ for $m\geq 1$ have nontrivial centers $Z$, and we may also apply Theorem~\ref{SemidirectTheorem}
with $L=Z$. For $\mathbb{R}^{n}\rtimes_\rho\operatorname{GL}_{n}(\mathbb R)$ with $n\geq 2$ we
have $L=\RR^*$.

It remains to consider the cases of the groups
of $\RR^n\rtimes_\rho \operatorname{SO}(n)$, for $n\geq 5$ odd, and 
${\RR^n\rtimes_\rho\operatorname{SL}_n(\mathbb R)}$, for $n\geq 3$ odd.
Put $H(n)=\operatorname{SO}(n)$ or $H(n)=\operatorname{SL}_n(\mathbb R)$
and ${G=\mathbb R^n\rtimes_\rho H(n)}$, 
where $\rho:H(n)\xhookrightarrow{\ \ \ }\operatorname{GL}_n(\mathbb R)$ is the natural representation.
Suppose that $n$ is odd. We decompose the matrices in $H(n)$ into $2\times2$ block matrices
$\big(\begin{smallmatrix} a & b\\ c & d \end{smallmatrix}\big)$, with $a$ of size $(n-1)\times(n-1)$.
Put $z=\big(\begin{smallmatrix} -1 & 0\\ \ \,0 & 1\end{smallmatrix}\big)\in H(n)$.
The $G$-centralizer of $z$
consists of all pairs $(w,g)$, where 
$w\in\mathbb R^n$ is a vector whose first $n-1$
entries are $0$, and $g=\big(\begin{smallmatrix} a & 0\\ 0 & d\end{smallmatrix}\big)$.
Put 
\[L=\big\{ \big(\begin{smallmatrix} a & 0\\ 0 & d\end{smallmatrix}\big)\in H(n)\big\}.\]
Since $n$ is odd, $d^{-1/n}$ exists in $\mathbb R$, and $\det(d^{-1/n} a)=1$.
The group $H(n-1)$ injects into $H(n)$ as the group of block diagonal matrices
$b=\big(\begin{smallmatrix} b & 0\\ 0 & 1\end{smallmatrix}\big)$. Every element in $L$ is thus a product
of a matrix in $H(n-1)$ and a diagonal matrix of the form 
$\big(
\begin{smallmatrix} 
d^{-1/n}  & 0\\ 0 & d 
\end{smallmatrix}\big)$.
By our assumptions on $n$, the group $H(n-1)$ has a
semisimple Lie algebra whose simple ideals are absolutely simple.
Note that here the assumption $n\geq 5$ enters in the orthogonal case.
Now we choose $h$ and $X$ in $H(n-1)$ as in Proposition~\ref{SemisimpleProp}.
Then 
\begin{align*}
C &=\{ghg^{-1}\mid g\in\operatorname{Cen}_{H(n-1)}(X)\}\\
&=\{ghg^{-1}\mid g\in\operatorname{Cen}_{L}(X)\}\\
&=\{ghg^{-1}\mid g\in\operatorname{Cen}_{H(n)}(\{z\}\cup X)\}\\
&=\{ghg^{-1}\mid g\in\operatorname{Cen}_{G}(\{z\}\cup X)\}.
\end{align*}
The last equality follows since $h$, being an element of $H(n-1)$, fixes every vector $w$
whose first $n-1$ coordinates are $0$.
The set $C$ is compact and contains a nonconstant path.
Since $\operatorname{Lie}(H(n))$ is absolutely simple, $C$ is spacious in $H(n)$ by
Proposition~\ref{AbsolutelySimpleProp}.
By Lemma~\ref{SemiLemma}, the set $C$ is spacious in $G$.

Suppose that $K$ is a group in $\mathbfcal K$ and that $\phi:K\longrightarrow G$ is
an abstract surjective homomorphism whose kernel $N$ is $\mathbfcal K$-analytic.
By Lemma~\ref{LittleLemma}, the set $\phi^{-1}(C)$ is
$\mathbfcal K$-analytic. Hence $G$ is rigid by Theorem~\ref{MainThm}.
\end{proof}

\section{Rigidity of finitely generated profinite groups}

We recall the definition of a verbal subgroup in a group $G$.
\subsection{Verbal subgroups}
Let $G$ be a group and let $F_s$ denote the free group on $s$ generators $x_1,\ldots,x_s$.
By the universal property of $F_s$, every $s$-tuple
$\mathbf g=(g_1,\ldots,g_s)$ of elements of $G$ determines a unique homomorphism
$\phi_{\mathbf g}:F_s\longrightarrow G$, with $\phi_{\mathbf g}(x_i)=g_i$. 
Let $w\in F_s$ be a word, i.e. an element in the free group. 
The \emph{word map} \[w(-):G\times \cdots\times G\longrightarrow G\]
maps $\mathbf g $ to $\phi_{\mathbf g}(w)$. To put it differently, we express $w$ as a word 
$w=w(x_1,\ldots,x_s)$ in
the generators $x_1,\ldots,x_s$
and then we substitute $g_i$ for $x_i$ in $w$,
for $i=1,\ldots,s$, and consider the resulting element $w(g_1,\ldots,g_s)$ in $G$. 
The image of the word map is the \emph{verbal set}
\[
G^w=\{\phi_{\mathbf g}(w)\mid\mathbf g\in G\times\cdots\times G\}
=\{w(g_1,\ldots,g_{s})\mid g_1,\ldots,g_s\in G\}.
\]
The corresponding \emph{verbal subgroup} is 
\[
 w(G)=\langle G^w\rangle.
\]
For example if $s=2$ and $w=[x_1,x_2]$, then $G^w$ is the set of all commutators in $G$ and 
$w(G)$ is the commutator group of $G$.

A topological group $G$ is called \emph{topologically finitely generated} if there exists
a finitely generated dense subgroup in $G$. 
For the case of compact groups, this is also called a compact group of \emph{finite generating rank}
in \cite[Definition 12.15]{HMCompact}.
A profinite group which is topologically finitely generated  is second countable
by \cite[Proposition 12.28]{HMCompact} and in particular metrizable by 
\cite[Corollary~A4.19,~p.~838]{HMCompact}. 

Nikolov--Segal proved
that in a topologically finitely generated profinite group $G$ every abstract subgroup of
finite index is open \cite[Theorem 1.1]{NikolovSegal}. In the course of their proof they show implicitly
that every open subgroup of $G$ contains a verbal open subgroup; this is worked out
in detail in \cite[Lemma~15]{Pejic}. Using this result, it is shown in \cite[Theorem 16]{Pejic}
that every topologically finitely generated profinite group $G$ carries a \emph{unique}
Polish group topology. We generalize this below. First we need an auxiliary result.

\begin{Lem}\label{LittleLemma2}
Let $\mathbfcal K$ be an almost Polish class, let
$G$ be a topological group  in $\mathbfcal K$, with a normal
(but not necessarily closed) subgroup $N\unlhd G$ which is $\mathbfcal{K}$-analytic.
Put $H=G/N$ (as an abstract group) and let $\pi:G\longrightarrow H$ denote 
the quotient map.
Let $w$ be a word in $F_s$.
Then $\pi^{-1}(H^w)$ is $\mathbfcal K$-analytic and $\pi^{-1}(w(H))$ is almost open.
\end{Lem}
\begin{proof}
We proceed similarly as in the proof of Lemma~\ref{LittleLemma}.
The set
\[X=\pi^{-1}(H^w)=\{w(g_1,\ldots,g_s)n\mid g_i\in G,n\in N\}=G^wN\]
is $\mathbfcal K$-analytic by Lemma~\ref{LittleLemma}(1), and so is $X^{-1}$.
Again by Lemma~\ref{LittleLemma}(1), each term in the countable union
\[\pi^{-1}(w(H))=\bigcup_{n\geq 1}(XX^{-1})^{\cdot n}\]
 is $\mathbfcal K$-analytic.
Since $\mathbfcal K$ is almost Polish, the $\mathbfcal K$-analytic sets $(XX^{-1})^{\cdot n}$ are almost open. 
Since the almost open sets form a
$\sigma$-algebra, $\pi^{-1}(w(H))$ is also almost open.
\end{proof}
We have the following generalization of \cite[Theorem 16]{Pejic}, which shows in particular that 
a topologically finitely generated profinite group
is rigid within the classes $\mathbfcal L^\sigma$, $\mathbfcal C$ and $\mathbfcal P$.
This is Theorem~C from the introduction.
\begin{Thm}\label{ProfiniteTheorem}
Let $G$ be a topologically finitely generated profinite group. Let $\mathbfcal K$
be an almost Polish class. If $G$ is contained in $\mathbfcal K$, then $G$ is
rigid within $\mathbfcal K$.
\end{Thm}
\begin{proof}
Suppose that $1\longrightarrow N\xhookrightarrow{\ \ \ }K\xrightarrow{\ \phi\ }G\longrightarrow 1$
is a short exact sequence of groups, where $\phi$ is an abstract group homomorphism, 
and that $(K,N)$ is in $\mathbfcal K_a$.
By \cite[Lemma~15]{Pejic}, every identity neighborhood $U\subseteq G$
contains an open verbal subgroup $w(G)$, for some word $w$ in some free group $F_s$.
By Lemma~\ref{LittleLemma2}, the preimage $\phi^{-1}(w(G))$ is almost open in $K$.
Hence $\phi$ is continuous and open by Theorem \ref{ContinuousOpenTheorem}.
\end{proof}

\section{Rigidity of compact semisimple groups}
\label{CompactSection}
A compact connected group is called \emph{semisimple} if it coincides with its
commutator group, see \cite[Statements~9.4, 9.5, 9.6]{HMCompact}. 
We recall the structure theorem for compact semisimple groups from \cite[Theorem 9.19]{HMCompact}.
Many more structural results about such groups can be found in Chapter 9 in \emph{loc.cit.}
Let us call a Lie group \emph{almost simple} if it is connected and if its Lie algebra is simple.
We call a group homomorphism \emph{central} if its kernel is contained in the center of its domain.
\begin{Thm}\label{StructureThm}\cite[Theorem 9.19 and Theorem 9.2]{HMCompact}
Let $G$ be a compact connected semisimple group. Then there exists a family of compact
almost simple Lie groups $(S_i)_{i\in I}$ and a central continuous open surjective
homomorphism \[\rho:\prod_{i\in I}S_i\longrightarrow G.\]
Moreover, every element in $G$ is a commutator.
\end{Thm}
It turns out that for proving automatic continuity of abstract homomorphisms onto 
compact connected semisimple groups, we need only the abstract group-theo\-re\-tic properties
stated in Theorem~\ref{StructureThm}. This abstraction allows us to prove a
continuity result which applies also to a wide class of profinite groups.
We set the stage as follows.
\subsection{Quasisimple and quasi-semisimple groups}\label{Quasi-Def}
Extending widespread terminology from finite group theory, we call a nontrivial compact group
$S$ \emph{quasisimple} if its abstract commutator group is dense and if $S/\operatorname{Cen}(S)$ 
is topologically simple (meaning that $S/\operatorname{Cen}(S)$ is nontrivial and has no nontrivial closed proper normal subgroups).

We call a compact group $G$ \emph{quasi-semisimple}
if there exists a family of compact quasisimple
groups $(S_i)_{i\in I}$ and a continuous surjective central homomorphism
\[
\prod_{i\in I}S_i\xrightarrow{\ \rho\ } G. \leqno{(*{*}*)}
\]
Hence every compact connected semisimple group is quasi-semisimple.
We note that $\rho$ is automatically open, either by the Open Mapping Theorem 
\cite[II.5.29]{HewittRoss},  \cite{HoMoOpenMapping},  \cite[6.19]{Stroppel},
or alternatively by Theorem~\ref{ContinuousOpenTheorem}, applied 
to the almost Polish class 
$\mathbfcal C$ of compact spaces.
Thus every compact connected semisimple group is quasi-semisimple. But also profinite groups
like $\prod_{p\in\mathbb P}\operatorname{SL}_2(\mathbb F_p)$ are quasi-semisimple, 
where $\mathbb P$ denotes the set of all primes and $\mathbb F_p$ the field of $p$ elements.

Let $\rho$ be a homomorphism as in $(*{*}*)$ above.
For a subset $J\subseteq I$ we put $S_J=\prod_{j\in J}S_j$,
with the convention that $S_\emptyset=\{1\}$. 
We view the group $S_J$ as a compact 
subgroup of $S_I$, and we put $G_J=\rho(S_J)$. Thus $G_J$ is a compact
quasi-semisimple
subgroup of $G=G_I$. If $I=J\sqcup K$ is a partition of $I$, then $G_J$ and $G_K$ commute, whence
$G=G_JG_K$ and therefore $G_J\cap G_K\subseteq\operatorname{Cen}(G)$.

The next proposition clarifies the structure of compact quasisimple groups. It depends
heavily on several deep results.
\begin{Prop}\label{quasi-simpleProp}
Suppose that $S$ is a compact quasisimple group. Then either $S$ is a finite quasisimple
group, or $S$ is a compact almost simple Lie group. Every element in $S$ is a product of 
(at most) $2$ commutators.
\end{Prop}
\begin{proof}
Put $H=S/\operatorname{Cen}(S)$.
 By \cite[9.90]{HMCompact}, the group $H$ is simple as an abstract group,
and $H$ is a compact Lie group (first paragraph of the proof in \emph{loc.cit.}).
The identity component $H^\circ$ is thus either trivial, or $H^\circ=H$.
Let $S'\subseteq S$ denote the abstract commutator group of $S$.
Note that $S'$ is not contained in the center of $S$, since otherwise $S=\overline{S'}$ would be
abelian, and then $S'$ and hence $S$ would be trivial.
Therefore $S'$ surjects onto the simple group $H$.
In particular, $H$ is perfect and thus nonabelian.

If $H^\circ=\{1\}$, then $H$, being a compact Lie group, is finite and Schur's 
Theorem \cite[10.1.4]{Robinson} implies that
$S$ is finite and perfect. By the deep result \cite{Liebeck}, every element in $S$ is the product of
at most $2$ commutators.

If $H^\circ=H$, then $H$ is in particular a compact almost simple Lie group. 
We put $Z=\operatorname{Cen}(S)$ and we claim that $S$ is connected.
The image of $S^\circ$ in $H$ is a closed normal subgroup and hence either trivial,
or it coincides with $H$. In the first case, $S^\circ\subseteq Z$. But then 
the compact totally disconnected group $S/S^\circ$ maps onto $H$ through an open
homomorphism, contradicting  the fact that $H$ is connected.
Hence $S^\circ$ maps onto $H$, that is, $S^\circ Z=S$.
But then every commutator of elements in $S$ is contained in $S^\circ$, and therefore $S=S^\circ$
is a compact connected semisimple group \cite[9.3 and 9.5]{HMCompact}.
From the Structure Theorem \ref{StructureThm} and the fact that $S/Z$ is almost simple,
we conclude that $S$ is an almost simple Lie group.
By Got\^o's Theorem \cite[9.2]{HMCompact}, every element in $S$ is a commutator.
\end{proof}
\begin{Cor}
In a quasi-semisimple group, every element is a product of $2$ commutators.
\end{Cor}
For the next two lemmas we assume that $G$ is quasi-semisimple as in
Definition \ref{Quasi-Def}, that $(S_i)_{i\in I}$ is a family of compact
quasisimple groups, and that $\rho:\prod_{i\in I}S_i\longrightarrow G$ is
a central surjective continuous homomorphism. Note that $\rho$ is open
by the remark above.
\begin{Lem}\label{Lem-1}
If $I=J\sqcup K$ is a partition of $I$, then
\[\operatorname{Cen}_G(G_K)=G_J\operatorname{Cen}(G_K)\text{ and }
 G_J=\{[g_1,g_2][h_1,h_2]\mid g_1,g_2,h_1,h_2\in \operatorname{Cen}_G(G_K)\}.
\]
\end{Lem}
\begin{proof}
Every element $g\in G$ can be written as $g=ab$, with $a\in G_J$ and $b\in G_K$.
Let $h\in G_K$. Then $abh=hab$ holds if and only if $bh=hb$. Hence the centralizer
of $G_K$ consists of all elements of the form $ab$, with $a\in G_J$ and $b\in\operatorname{Cen}(G_K)$.
The commutator of two such elements is $[a_1b_1,a_2b_2]=[a_1,a_2][b_1,b_2]=[a_1,a_2]\in G_J$.
Since $G_J$ is quasi-semisimple, every element in $G_J$ is a product of two commutators of elements of $G_J$.
\end{proof}
\begin{Lem}\label{Lem-2}
Let $U\subseteq G$ be an identity neighborhood. Then there exists a finite subset
$J\subseteq I$, and an identity neighborhood $W$ in $G_J$ such that 
$WG_{I\setminus J}\subseteq U$ is an identity neighborhood in $G$.
\end{Lem}
\begin{proof}
By the definition of the product topology and from the continuity of $\rho$, we find
a finite index set $J\subseteq I$ and an identity neighborhood $W'\subseteq S_J$
such that $W'S_{I\setminus J}\subseteq S_I$ is an identity neighborhood, with
$\rho(W'S_{I\setminus J})\subseteq U$. Since $\rho:S_J\longrightarrow G_J$ is open, the set
$W=\rho(W')$ is an identity neighborhood in $G_J$, and 
similarly $\rho(W'S_{I\setminus J})=WG_{I\setminus J}$ 
is an identity neighborhood in $G$.
\end{proof}
The next result implies Theorem~B in the introduction. Note, however, that 
Theorem~\ref{CompactSemisimpleThm} applies also to profinite quasi-semisimple groups,
which need not be topologically finitely generated, separable, or metrizable.
\begin{Thm}\label{CompactSemisimpleThm}
Let $G$ be a compact quasi-semisimple group, and let 
$\mathbfcal K$ be an almost Polish class. If $G$ belongs to $\mathbfcal K$, then
$G$ is rigid within $\mathbfcal K$.
\end{Thm}
\begin{proof}
Suppose that \[1\longrightarrow N\xhookrightarrow{\ \ \ }K\xrightarrow{\ \phi\ }G\longrightarrow 1\]
is a short exact sequence of groups, where $\phi$ is an abstract group homomorphism, 
and that $(K,N)\in\mathbfcal K_a$.
We use the notation that we have set up above. Let $U\subseteq G$ be an identity neighborhood.
We choose $J\subseteq I$ finite and an identity neighborhood
$W\subseteq G_J$ as in Lemma \ref{Lem-2} above,
such that $WG_{I\setminus J}\subseteq U$.
Since $G_J$ is a compact Lie group, there exists a finite subset $X\subseteq G_J$ which
generates a dense subgroup in $G_J$, see Lemma \ref{tfg}.
By Lemma \ref{Lem-1}, \[G_{I\setminus J}=\{[g_1,g_2][h_1,h_2]\mid g_1,g_2,h_1,h_2\in \operatorname{Cen}_G(X)\}.\]
Then
$\phi^{-1}(G_{I\setminus J})$ 
is 
$\mathbfcal K$-analytic by Lemma \ref{LittleLemma}.
The Lie algebra of $G_J$ is semisimple.
Hence  we may choose elements $a_1,\ldots, a_r$ in the compact Lie group $G_J$
as in Theorem~\ref{vdW2}, such that
\[V=[a_1,G_J]\cdots[a_r,G_J]=[a_1,G]\cdots[a_r,G]\subseteq W\] is a compact identity
neighborhood in $G_J$. Then 
$VG_{I\setminus J}\subseteq WG_{I\setminus J}\subseteq U$ is a compact identity neighborhood, and 
$\phi^{-1}(VG_{I\setminus J})$ is $\mathbfcal K$-analytic by Lemma \ref{LittleLemma}.
The claim follows now from Theorem~\ref{ContinuousOpenTheorem}.
\end{proof}

\section{The proof of Theorem D}

In this last section we consider abstract homomorphisms $\psi:G\longrightarrow H$,
where $H$ is a topological group and $G$ is a Lie group whose Lie algebra is
perfect. 
The following result is Theorem~D from the introduction.
It generalizes some of the the main results in \cite{Shtern2006}.
\begin{Thm}\label{OtherDircetion}
Let $G$ be a Lie group whose Lie algebra is perfect. Let $H$ be a topological group, and let
$\psi:G\longrightarrow H$ be an abstract homomorphism. If there exists a
compact spacious set $C\subseteq G$ whose image $\psi(C)$ has compact closure, then 
$\psi$ is continuous.
\end{Thm}
\begin{proof}
We follow the strategy of \cite[Theorem 5.64]{HMCompact}.
Let \[D=g_1CC^{-1}\cdots g_{r}CC^{-1}\subseteq G\] be a compact identity neighborhood.
It follows that $\psi(D)\subseteq H$ has compact closure $E=\overline{\psi(D)}$.
Put $n=\dim(G)$.
Let $U\subseteq H$ be an arbitrary identity neighborhood.
By Wallace's Lemma \ref{Wallace}, there exists an identity neighborhood 
$W\subseteq H$ such that $[h_1,E]\cdots[h_n,E]\subseteq U$
for all $h_1,\ldots,h_n\in W$.

We choose $1$-parameter groups $c_1,\ldots,c_n$ in $G$ as in Theorem~\ref{vdW2}.
We claim that we can find numbers $t_i$ with $0<|t_i|\leq 1$ such that $h_i=\psi(c_i(t_i))\in W$.
Once we manage to do this, we have 
$\psi([c_1(t_1),D]\cdots[c_n(t_n),D])\subseteq[h_1,E]\cdots[h_n,E]\subseteq U$
and the continuity of $\psi$ at the identity follows, because $[c_1(t_1),D]\cdots[c_n(t_n),D]$ is 
by Theorem~\ref{vdW2} an identity neighborhood.
Then the global continuity of $\psi$ follows \cite[III. Proposition~23]{BourbakiTopology}.

Fix $i$ and put $\tau=\psi\circ c_i:\mathbb R\longrightarrow H$. 
There exists a number $0<s<1$ such that $c_i([0,s])\subseteq D$.
The interval $P=[0,s]$ generates $(\mathbb R,+)$ as a group, whence 
$Q=\tau(P)\subseteq E$ generates the group
$A=\tau(\mathbb R)$. Since $A$ is divisible, we have either 
$A=\{1\}=Q$, or $Q$ is infinite, because the only finitely
generated divisible abelian group is the trivial group. 

If $A=\{1\}$ put $t_i=1$.

Otherwise,
$Q$ is infinite, whence $Q\subseteq E$ has an accumulation point $h\in E$, because $E$ is compact. 
Let $V\subseteq H$ be an open identity neighborhood such that $VV^{-1}\subseteq W$.
Then $Vh$ contains infinitely many elements of $Q$.
We choose numbers $a,b$ with $0\leq a<b\leq s$ such that $\tau(a),\tau(b)\in Vh$.
Then $0<b-a\leq 1$ and $\tau(b-a)=\tau(b)\tau(a)^{-1}\in Vh(Vh)^{-1}=VV^{-1}\subseteq W$.
Hence we may put $t_i=b-a$.
\end{proof}
\begin{Cor}\cite[Theorem~5.64]{HMCompact}
Let $G$ be a Lie group whose Lie algebra is perfect, let $H$ be a compact group and
let $\psi:G\longrightarrow H$ be an abstract homomorphism. 
Then $\psi$ is continuous.
\end{Cor}

\begin{Cor}
Let $G$ be a Lie group whose Lie algebra is simple, let $H$ be a topological group and
let $\psi:G\longrightarrow H$ be an abstract homomorphism. If there exists
a nonconstant path $f:[0,1]\longrightarrow G$ such that $\psi(f([0,1]))\subseteq H$ has compact closure,
then $\psi$ is continuous.
\end{Cor}

\section{An Erratum and a Comment}
\label{Erratum}

We take the opportunity to correct a mistake which occurred in \cite{KramerUnique}.
In Theorem 7 and Corollary 8 in \emph{loc.cit.}, the hypothesis has to be added that 
the  kernel of the homomorphism $\phi$ is 
$\sigma$-compact.\footnote{\ The 
mistake occurs on p.~2625 where it is implicitly assumed that the preimage of a conjugacy class
under a homomorphism is again a conjugacy class, which need not be the case.}
The same assumption has to be added in Proposition~37 and in Proposition~41 in \cite{Braun}.
Example~\ref{Ex3} in the present article
shows that this hypothesis on the kernel cannot be 
omitted. Of course this hypothesis on the kernel is satisfied if $\phi$ is bijective. 
Theorem~11 in \cite{KramerUnique} is therefore not affected by the mistake.
However, this result is superseded by 
Theorem~\ref{SemisimpleThm} in the present article, which is more general.
There is also a small misprint in the statement of Theorem 3 in \cite{KramerUnique},
which should read \emph{`if $G$ is $\sigma$-compact, then\ldots'.}
This has no consequences, and was kindly pointed out by Ruppert McCallum.

We finally remark that the first paragraphs in \cite{Shtern2} might give the erroneous 
impression that
\cite[Theorem 18]{KramerUnique} is incorrect, and that \cite[Theorem 11]{KramerUnique}
is implied by the author's  earlier work  on locally bounded homomorphisms.


\begin{thebibliography}{YMCA}


\bibitem[TK]{Kal16}
W. M. Al-Tameemi\ and\ R. R. Kallman, The natural semidirect product $\mathbb{R}\sp n\rtimes G(n)$ is algebraically determined, Topology Appl. {\bf 199} (2016), 70--83. 

\bibitem[BT]{BT}
A. Borel\ and\ J. Tits, Homomorphismes ``abstraits'' de groupes alg\'ebriques simples, Ann. of Math. (2) {\bf 97} (1973), 499--571. 

\bibitem[Bou]{BourbakiTopology}
N. Bourbaki, {\it General topology. Chapters 1--4}, translated from the French, reprint of the 1966 edition, Elements of Mathematics (Berlin), Springer, Berlin, 1989. 

N. Bourbaki, {\it General topology. Chapters 5--10}, translated from the French, reprint of the 1966 edition, Elements of Mathematics (Berlin), Springer, Berlin, 1989. 

\bibitem[Br]{Braun}
O. Braun, {\em Uniqueness of topologies on compact connected groups}. Diploma Thesis, Univ. M\"unster 2016. 

\bibitem[Car]{Car}
\'E. Cartan, Sur les repr\'esentations lin\'eaires des groupes clos, Comment. Math. Helv. {\bf 2} (1930), no.~1, 269--283. 

\bibitem[DD]{DielsDowerk}
L. Diels\ and\ P. A. Dowerk,
Invariant automatic continuity for compact connected simple Lie groups,
preprint,  arXiv:1811.04618v1 [math.GR] (2018).

\bibitem[DT]{DowerkThom}
P. A. Dowerk\ and\ A. Thom,
Bounded normal generation and invariant automatic continuity,
preprint, arXiv:1506.08549v2 [math.OA] (2015).

\bibitem[Dug]{Dug}
J. Dugundji, {\it Topology}, Allyn \&\ Bacon, Boston, Mass., 1966. 

\bibitem[Freu]{Freu}
H. Freudenthal, Die Topologie der Lieschen Gruppen als algebraisches Ph\"anomen. I, Ann. of Math. (2) {\bf 42} (1941), 1051--1074. 

\bibitem[FdV]{FdV}
H. Freudenthal\ and\ H. de Vries, {\it Linear Lie groups}, Pure and Applied Mathematics, Vol. 35, Academic Press, New York, 1969. 

\bibitem[GP]{Pejic}
P. Gartside\ and\ B. Peji\'c, Uniqueness of Polish group topology, Topology Appl. {\bf 155} (2008), no.~9, 992--999. 

\bibitem[Glus]{Gluskov}
V. M. Glu\v skov, Structure of locally bicompact groups and Hilbert's fifth problem, Uspehi Mat. Nauk (N.S.) {\bf 12} (1957), no.~2 (74), 3--41. 
\\
Translated as:
\\
V. M. Glu\v skov, 
The structure of locally compact groups and Hilbert's fifth problem, 
Amer. Math. Soc. Transl. (2) {\bf 15} (1960), 55--93. 

\bibitem[Go]{Goto}
M. Goto, On an arcwise connected subgroup of a Lie group, Proc. Amer. Math. Soc. {\bf 20} (1969), 157--162. 

\bibitem[G\"un]{HG}
H. G\"undo\u gan, The component group of the automorphism group of a simple Lie algebra and the splitting of the corresponding short exact sequence, J. Lie Theory {\bf 20} (2010), no.~4, 709--737. 

\bibitem[HHM]{HeHoMo}
S. Hern\'andez, K. H. Hofmann\ and\ S. A. Morris, Nonmeasurable subgroups of compact groups, J. Group Theory {\bf 19} (2016), no.~1, 179--189. 

\bibitem[HR]{HewittRoss}
E. Hewitt\ and\ K. A. Ross, {\it Abstract harmonic analysis. Vol. I}, Second edition, Springer, Berlin, 1979. 

\bibitem[HHL]{HiHoLa}
J. Hilgert, K. H. Hofmann\ and\ J. D. Lawson, {\it Lie groups, convex cones, and semigroups}, Oxford Mathematical Monographs, Oxford Univ. Press, New York, 1989. 

\bibitem[HN]{HilgertNeeb}
J. Hilgert\ and\ K.-H. Neeb, {\it Structure and geometry of Lie groups}, Springer Monographs in Mathematics, Springer, New York, 2012. 

\bibitem[HK]{HofmannKramer}
K. H. Hofmann\ and\ L. Kramer, Transitive actions of locally compact groups on locally contractible spaces, 
J. Reine Angew. Math. {\bf 702} (2015), 227--243. 

Erratum, J. Reine Angew. Math. {\bf 702} (2015), 245--246. 

\bibitem[HM1]{hofmori}
K. H. Hofmann\ and\ S. A. Morris, Transitive actions of compact groups and topological dimension, J. Algebra {\bf 234} (2000), no.~2, 454--479. 

\bibitem[HM2]{HoMoOpenMapping}
K. H. Hofmann\ and\ S. A. Morris, Open mapping theorem for topological groups, Topology Proc. {\bf 31} (2007), no.~2, 533--551. 

\bibitem[HM3]{HoMoAlmost}
K. H. Hofmann\ and\ S. A. Morris, The structure of almost connected pro-Lie groups,
J. of Lie Theory {\bf21} (2011), 341--383. 

\bibitem[HM4]{HMCompact}
K. H. Hofmann\ and\ S. A. Morris, {\it The structure of compact groups}, third edition, revised and augmented., De Gruyter Studies in Mathematics, 25, de Gruyter, Berlin, 2013. 

\bibitem[HM5]{HoMoAxioms}
K. H. Hofmann\ and\ S. A. Morris,  Pro-Lie Groups: A Survey with Open Problems,
Axioms {\bf4} (2015), 294--312.

\bibitem[Iwa]{iwa}
K. Iwasawa, On some types of topological groups, Ann. of Math. (2) {\bf 50} (1949), 507--558. 

\bibitem[Kal1]{Ka74}
R. R. Kallman, The topology of compact simple Lie groups is essentially unique, Advances in Math. {\bf 12} (1974), 416--417. 

\bibitem[Kal2]{Kal82}
R. R. Kallman, A uniqueness result for a class of compact connected groups, in {\it Conference in modern analysis and probability (New Haven, Conn., 1982)}, 207--212, Contemp. Math., 26, 
Amer. Math. Soc., Providence, RI. 

\bibitem[Kec]{Kechris}
A. S. Kechris, {\it Classical descriptive set theory}, Graduate Texts in Mathematics, 156, Springer, New York, 1995. 


\bibitem[Kil]{Kiltinen}
J. O. Kiltinen, On the number of field topologies on an infinite field, Proc. Amer. Math. Soc. {\bf 40} (1973), 30--36. 
 
\bibitem[Kra]{KramerUnique}
L. Kramer, The topology of a semisimple Lie group is essentially unique, Adv. Math. {\bf 228} (2011), no.~5, 2623--2633. 

\bibitem[Kur]{Kur}
K. Kuratowski, {\it Topology. Vol. I}, New edition, revised and augmented. Translated from the French by J. Jaworowski, Academic Press, New York, 1966. 

\bibitem[LBST]{Liebeck}
M. W. Liebeck\ et al., Commutators in finite quasisimple groups, Bull. Lond. Math. Soc. {\bf 43} (2011), no.~6, 1079--1092. 

\bibitem[MZ]{MoZi}
D. Montgomery\ and\ L. Zippin, {\it Topological transformation groups}, Interscience Publishers, New York, 1955. 

\bibitem[Mur]{Murakami}
S. Murakami, On the automorphisms of a real semi-simple Lie algebra, J. Math. Soc. Japan {\bf 4} (1952), 103--133. 

\bibitem[NS]{NikolovSegal}
N. Nikolov\ and\ D. Segal, On finitely generated profinite groups. I. Strong completeness and uniform bounds, Ann. of Math. (2) {\bf 165} (2007), no.~1, 171--238. 

\bibitem[Oxt]{Oxtoby}
J. C. Oxtoby, {\it Measure and category}, second edition, Graduate Texts in Mathematics, 2, Springer, New York, 1980. 

\bibitem[Pet]{Pettis}
B. J. Pettis, On continuity and openness of homomorphisms in topological groups, Ann. of Math. (2) {\bf 52} (1950), 293--308. 

\bibitem[Rob]{Robinson}
D. J. S. Robinson, {\it A course in the theory of groups}, second edition, Graduate Texts in Mathematics, 80, Springer, New York, 1996. 

\bibitem[Sht1]{Shtern2006}
A. I. Shtern, Van der Waerden continuity theorem for semisimple Lie groups, Russ. J. Math. Phys. {\bf 13} (2006), no.~2, 210--223. 

\bibitem[Sht2]{Shtern2}
A. I. Shtern, Bounded structure and continuity for homomorphisms of perfect connected locally compact groups, Proc. Jangjeon Math. Soc. {\bf 15} (2012), no.~3, 235--240. 

\bibitem[Ste]{Stewart}
T. E. Stewart, Uniqueness of the topology in certain compact groups, Trans. Amer. Math. Soc. {\bf 97} (1960), 487--494. 

\bibitem[Str]{Stroppel}
M. Stroppel, {\it Locally compact groups}, EMS Textbooks in Mathematics, European Mathematical Society (EMS), Z\"urich, 2006. 

\bibitem[Ti]{Tits}
J. Tits, Homorphismes ``abstraits'' de groupes de Lie, in {\it Symposia Mathematica, Vol. XIII (Convegno di Gruppi e loro Rappresentazioni, INDAM, Rome, 1972)}, 479--499, Academic Press, London. 

\bibitem[vdW]{vdW}
B. L. van der Waerden, Stetigkeitss\"atze f\"ur halbeinfache Liesche Gruppen, Math. Z. {\bf 36} (1933), no.~1, 780--786. 

\bibitem[War]{Warner}
F. W. Warner, {\it Foundations of differentiable manifolds and Lie groups}, corrected reprint of the 1971 edition, Graduate Texts in Mathematics, 94, Springer, New York, 1983. 

\bibitem[Wrn]{WarnerHarmonic}
G. Warner, {\it Harmonic analysis on semi-simple Lie groups. I}, Springer, New York, 1972. 

\end{thebibliography}
\end{document}